\RequirePackage[english]{babel}
\documentclass[a4paper,twoside,11pt]{amsart}[2009/07/02] 
\usepackage[latin1]{inputenc}
\usepackage{amsmath}
\usepackage{amssymb}
\usepackage{amsfonts}
\usepackage{fixltx2e}[2005/12/01]

\usepackage{enumerate}

\renewcommand{\theenumii}{{\upshape{(\alph{enumii})}}}

\newcommand{\myitemii}[1]{\renewcommand{\theenumii}{{\upshape{(#1)}}}\item}

\usepackage[matrix,arrow,cmtip,frame,graph]{xy} 
\SelectTips{cm}{}
\newdir{(}{{}*!/-5pt/@^{(}}
\newdir{(x}{{}*!/-5pt/@_{(}}
\newdir{+}{{}*!/-9pt/{}}
\newdir{>+}{@{>}*!/-9pt/{}}
\entrymodifiers={+!!<0pt,\fontdimen22\textfont2>}

\usepackage[nosubsections]{mytheorems2}

\newcommand{\pref}[1]{\eqref{#1}}

\DeclareMathOperator{\Hom}{Hom}

\newcommand{\Z}{\mathbb{Z}}

\newcommand{\inj}{\hookrightarrow}
\newcommand{\iso}{\cong} 
\newcommand{\map}[3]{#1\,:\, #2\rightarrow #3}
\DeclareMathOperator{\image}{im}

\newcommand{\etale}{\'{e}tale}
\newcommand{\Etale}{\'{E}tale}

\newcommand{\Sch}{\mathbf{Sch}} 

\newcommand{\Spec}{\mathrm{Spec}}
\newcommand{\Min}{\mathrm{Min}}
\newcommand{\sep}{\mathrm{sep}}
\newcommand{\cons}{\mathrm{cons}} 
\newcommand{\genpts}{\mathrm{gen}} 
\newcommand{\red}{\mathrm{red}}
\newcommand{\shensel}[1]{{^{\mathrm{sh}}{#1}}}

\newcommand{\A}[1]{\mathbb{A}^{#1}}    

\newcommand\ip{\mathfrak{p}} 
\newcommand\im{\mathfrak{m}} 
\newcommand\sO{\mathcal{O}} 
\newcommand\sF{\mathcal{F}} 
\newcommand\sA{\mathcal{A}} 
\newcommand\sB{\mathcal{B}} 
\newcommand\id[1]{\mathrm{id}_{#1}}

\newcommand{\normcmd}[3]{%
#2%
\ifx\\#1\\%
^{\,\mathrm{#3}}%
\else%
^{#1/\mathrm{#3}}%
\fi%
}
\newcommand{\wn}[2]{\normcmd{#1}{#2}{wn}} 
\newcommand{\rwn}[2]{{^{\,*}_{#1}#2}} 

\newcommand{\catE}{\mathbf{E}}
\newcommand{\catF}{\mathbf{F}}
\newcommand{\sN}{\mathcal{N}}
\newcommand{\qc}{\mathrm{qc}}
\newcommand{\finite}{\mathrm{fin}}
\newcommand{\metale}{\text{\'et}} 
\newcommand{\interior}{\mathrm{int}}
\newcommand{\TIC}{\mathrm{TIC}}

\newcommand{\cU}{\mathcal{U}}
\newcommand{\cT}{\mathcal{T}} 
\newcommand{\equalizer}[2]{\xymatrix@1@M=0mm@C=7mm{#1%
\ar@<.5ex>@{+->+}[r] \ar@<-.5ex>@{+->+}[r] & #2}}

\newcommand\loccit{\textit{loc.\ cit.}}

\begin{document}

\title[Submersions and effective descent of \'etale morphisms]
{Submersions and effective descent\\of \'etale morphisms}
\author{David Rydh}
\address{Department of Mathematics, University of California, Berkeley,
970 Evans Hall, Berkeley, CA 94720-3840 USA}
\email{dary@math.berkeley.edu}
\date{2010-03-03}

\begin{abstract}
Using the flatification by blow-up result of Raynaud and Gruson, we obtain new
results for submersive and subtrusive morphisms. We show that universally
subtrusive morphisms, and in particular universally open morphisms, are
morphisms of \emph{effective} descent for the fibered category of \etale{}
morphisms. Our results extend and supplement previous treatments on submersive
morphisms by Grothendieck, Picavet and Voevodsky. Applications include the
universality of geometric quotients and the elimination of noetherian
hypotheses in many instances.
\end{abstract}

\maketitle


\setcounter{secnumdepth}{0}
\begin{section}{Introduction}
Submersive morphisms, that is, morphisms inducing the quotient topology on the
target, appear naturally in many situations such as when studying quotients,
homology, descent and the fundamental group of schemes. Somewhat unexpected,
they are also closely related to the integral closure of ideals. Questions
related to submersive morphisms of \emph{schemes} can often be resolved by
topological methods using the description of schemes as locally ringed spaces.
Corresponding questions for \emph{algebraic spaces} are significantly harder as
an algebraic space is not fully described as a ringed space. The main result of
this paper is an effective descent result which bridges this gap between
schemes and algebraic spaces.

The first proper treatment of submersive morphisms seems to be due to
Grothendieck~\cite[Exp.~IX]{sga1} with applications to the fundamental group of
a scheme. He shows that submersive morphisms are morphisms of descent for the
fibered category of \etale{} morphism. He then proves \emph{effectiveness} for
the fibered category of quasi-compact and separated \etale{} morphisms in some
special cases, e.g., for finite morphisms and universally open morphisms of
finite type between noetherian schemes. Our main result consists of several
very general effectiveness results extending those of Grothendieck
significantly. For example, we show that any universal submersion of
noetherian schemes is a morphism of effective descent for quasi-compact
\etale{} morphisms. As an application, these effectiveness results imply that
strongly geometric quotients are categorical in the category of algebraic
spaces~\cite{rydh_finite_quotients}.

Later on Picavet singled out a subclass of submersive morphisms
in~\cite{picavet_submersion}. He termed these morphisms \emph{subtrusive} and
undertook a careful study of their main properties. The class of subtrusive
morphisms is natural in many respects. For example, over a locally noetherian
scheme, every submersive morphism is subtrusive. Picavet has also given an
example showing that a finitely presented universally submersive morphism is
not necessarily subtrusive. In particular, not every finitely presented
universally \emph{submersive} morphism is a limit of finitely presented
submersive morphisms of noetherian schemes. We will show that every finitely
presented universally \emph{subtrusive} morphism is a limit of finitely
presented submersive morphisms of noetherian schemes. This is a key result
missing in~\cite{picavet_submersion} allowing us to eliminate noetherian
hypotheses in questions about universal subtrusions of finite presentation. It
also shows that the class of subtrusive morphisms is indeed an important and
very natural extension of submersive morphisms of noetherian schemes.

A general observation is that in the noetherian setting it is often useful to 
describe submersive morphisms using the subtrusive property.
For example, there is a valuative criterion for submersions of noetherian
schemes~\cite[Prop.~3.7]{kollar_quotients} which rather describes the essence
of the subtrusiveness.

\begin{subsection}{Structure theorem}
An important tool in this article is the structure theorems for universally
subtrusive morphisms given in~\S\ref{S:structure-theorems}: Let $\map{f}{X}{Y}$
be a universally subtrusive morphism of finite presentation. Then there is a
morphism $\map{g}{X'}{X}$ and a factorization of $f\circ g$
$$\xymatrix{X'\ar[r]^{f_1} & Y'\ar[r]^{f_2} & Y}$$
where $f_1$ is fppf and $f_2$ is proper, surjective and of finite presentation,
cf.\ Theorem~\pref{T:ST-aff}. This is shown using the flatification
result of Raynaud and Gruson~\cite{raynaud-gruson}.

We also show that if $f$ is in addition quasi-finite, then there is a similar
factorization as above such that $f_1$ is an open covering and $f_2$ is finite,
surjective and of finite presentation, cf.\ Theorem~\pref{T:STf-noeth/aff}.
Combining these results, we show that every universally subtrusive morphism of
finite presentation $\map{f}{X}{Y}$ has a refinement $X'\to Y$ which factors
into an open covering $f_1$ followed by a surjective and proper morphism of
finite presentation $f_2$.

This structure theorem is a generalization to the non-noetherian case of a
result of Voevodsky~\cite[Thm.~3.1.9]{voevodsky_homology}. The proof is
somewhat technical and the reader without any interest in non-noetherian
questions may prefer to read the proof given by Voevodsky which has a more
geometric flavor. Nevertheless, our extension is crucial for the elimination
of noetherian hypotheses referred to above.

As a first application, we show in Section~\ref{S:loc-closed} that universally
subtrusive morphisms of finite presentation are morphisms of effective descent
for locally closed subsets. This result is not true for universally
\emph{submersive} morphisms despite its topological nature.
\end{subsection}

\begin{subsection}{Effective descent of \'etale morphisms}
In Section~\ref{S:descent} we use the structure theorems
of~\S\ref{S:structure-theorems} and the proper base change theorem in \etale{}
cohomology to prove that
\begin{itemize}
\item Quasi-compact universally \emph{subtrusive} morphisms are morphisms of
effective descent for finitely presented \etale{}
morphisms, cf.\ Theorem~\pref{T:descent-for-etqc}.
\item Universally open and surjective morphisms are morphisms of effective
descent for \etale{} morphisms, cf.\ Theorem~\pref{T:descent-for-et}.
\end{itemize}
In particular, universal submersions between noetherian schemes are morphisms
of effective descent for quasi-compact \etale{} morphisms.
\end{subsection}

\begin{subsection}{Applications}
The effective descent results of~\S\ref{S:descent} have several applications.
One is the study of the algebraic fundamental group using morphisms of
effective descent for finite \etale{} covers, cf.~\cite[Exp.~IX, \S
5]{sga1}. Another application, also the origin of this paper, is in the theory
of quotients of schemes by groups. The effective descent results show that
strongly geometric quotients are \emph{categorical} in the category of
algebraic spaces~\cite{rydh_finite_quotients}. This result is obvious in the
category of schemes but requires the results of~\S\ref{S:descent} for the
extension to algebraic spaces. The third application in mind is similar to the
second. Using the effective descent results we can extend some basic results on
the $h$- and $qfh$-topologies defined by Voevodsky~\cite{voevodsky_homology} to
the category of algebraic spaces. This is done
in~\S\S\ref{S:descent-of-morphisms}--\ref{S:h-top}. The $h$-topology has been
used in singular homology~\cite{voevodsky-suslin_sing-hom}, motivic homology
theories~\cite{voevodsky_motives} and when studying families of
cycles~\cite{voevodsky_cycles}. The $h$-topology is also related to the
integral closure of ideals~\cite{brenner_Groth-top}.
\end{subsection}

\begin{subsection}{Elimination of noetherian hypotheses}
Let $S$ be an inverse limit of affine schemes $S_\lambda$. The situation in
mind is as follows. Every ring $A$ is the filtered direct limit of its
subrings $A_\lambda$ which are of finite type over $\Z$. The scheme
$S=\Spec(A)$ is the inverse limit of the excellent noetherian schemes
$S_\lambda=\Spec(A_\lambda)$.

Let $X\to S$ be a finitely presented morphism. Then $X\to S$ descends to a
finitely presented morphism $X_\lambda\to S_\lambda$ for sufficiently large
$\lambda$~\cite[Thm.~8.8.2]{egaIV}. By this, we mean that $X\to S$ is
the base change of $X_\lambda\to S_\lambda$ along $S\to S_\lambda$.
If $X\to S$ is proper (resp.\ flat, \etale{}, smooth, etc.) then so is
$X_\lambda\to S_\lambda$ for sufficiently large $\lambda$,
cf.~\cite[Thm.~8.10.5, Thm.~11.2.6, Prop.~17.7.8]{egaIV}. Note that the
corresponding result for universally open is missing in~\cite{egaIV}. As we
have mentioned earlier, the analogous result for universally submersive is
false.

In Theorem~\pref{T:subtrusive-limit-of-subtrusive} we show that if $X\to S$ is
\emph{universally subtrusive} then so is $X_\lambda\to S_\lambda$ for
sufficiently large $\lambda$. We also show the corresponding result for $X\to
S$ \emph{universally open}. An easy application of this result is the
elimination of noetherian hypotheses in~\cite[\S\S14--15]{egaIV}. In
particular, every universally open morphism locally of finite presentation has
a locally quasi-finite quasi-section, cf.~\cite[Prop.~14.5.10]{egaIV}.
\end{subsection}

\begin{subsection}{Appendices}
Some auxiliary results are collected in two appendices. In the first appendix
we recall the henselian properties of a scheme which is proper over a complete
or henselian local ring. These properties follow from the Stein factorization
and Grothendieck's existence theorem and constitute a part of the proper base
change theorem in \etale{} cohomology. With algebraic spaces we can express
these henselian properties in an appealing form which is used when proving the
effective descent results in Section~\ref{S:descent}.

In the second appendix, we briefly recall the weak subintegral closure of rings
and weakly normal extensions. We also introduce the \emph{absolute weak
normalization} which we have not found elsewhere. When $X$ is an integral
scheme, the absolute weak normalization is the weak subintegral closure in the
perfect closure of the function field of $X$. The absolute weak normalization
is used to describe the sheafification of a representable functor in the
$h$-topology.
\end{subsection}

\begin{subsection}{Terminology and assumptions}
A morphism of schemes or algebraic spaces is called a \emph{nil-immersion} if
it is a surjective immersion. Equivalently, it is a closed immersion given by
an ideal sheaf which is a nil-ideal, i.e., every section of the ideal sheaf is
locally nilpotent.

Given a covering $\map{f}{X}{Y}$ we say that $\map{f'}{X'}{Y}$ is a
\emph{refinement} of $f$ if
$f'$ is covering and factors through $f$.
For general terminology and properties of algebraic
spaces, see Knutson~\cite{knutson_alg_spaces}. As in~\cite{knutson_alg_spaces}
we assume that all algebraic spaces are quasi-separated.
\end{subsection}

\begin{subsection}{Acknowledgment}
The author would like to thank the referee for a very careful reading and
for several suggestions which improved the paper.
\end{subsection}

\end{section}
\setcounter{secnumdepth}{3}


\begin{section}{Topologies}

In addition to the Zariski topology, we will have use of two additional
topologies which we recall in this section. The first is the
\emph{constructible} topology, cf.~\cite[\S7.2]{egaI_NE}, which also is known
as the \emph{patch} topology. The second topology is the $S$-topology where $S$
stands for specialization. We then define submersive morphisms and give
examples of morphisms which are submersive in the constructible topology.

The closed (resp.\ open) subsets of the constructible topology are the
pro-constructible (resp.\ ind-constructible) subsets. A subset is
pro-constructible (resp.\ ind-constructible) if it locally is an intersection
(resp.\ union) of constructible sets. An important characterization of
pro-constructible subsets is given by the following proposition.

\begin{proposition}[{\cite[Prop.~7.2.1]{egaI_NE}}]\label{P:char-of-pro-constr}
Let $X$ be a quasi-compact and quasi-separated scheme. A subset $E\subseteq X$
is pro-constructible if and only if there is an affine scheme $X'$ and
a morphism $\map{f}{X'}{X}$ such that $E=f(X')$.
\end{proposition}

If $X$ is a scheme, then we denote by $|X|$ its underlying topological space
with the Zariski topology and $|X|^{\cons}$ its underlying topological space
with the constructible topology. If $\map{f}{X}{Y}$ is a morphism of schemes
then we let $f^{\cons}$ be the underlying map in the constructible topology.

\begin{proposition}[{\cite[Prop.~7.2.12]{egaI_NE}}]
\label{P:constructible-topology}
Let $X$ be a scheme.
\begin{enumerate}
\item If $\map{f}{X}{Y}$ is a morphism of schemes, then $f^{\cons}$ is
continuous.
\item If $\map{f}{X}{Y}$ is \emph{quasi-compact}, then $f^{\cons}$ is
closed.
\item If $\map{f}{X}{Y}$ is \emph{locally of finite presentation}, then
$f^{\cons}$ is open.
\item If $Z\inj X$ is closed, then ${|X|^{\cons}}_{|Z}=|Z|^{\cons}$.
\label{PI:CT-closed}
\item If $U\subseteq X$ is open, then ${|X|^{\cons}}_{|U}=|U|^{\cons}$.
\label{PI:CT-open}
\item If $W$ is a locally closed subscheme of $X$, then
${|X|^{\cons}}_{|W}=|W|^{\cons}$.
\label{PI:CT-loc-closed}
\end{enumerate}
\end{proposition}
\begin{proof}
(i)--(iii) are~\cite[Prop.~7.2.12~(iii)--(v)]{egaI_NE}. Statements (iv) and (v)
are consequences of (ii) and (iii) respectively, as closed immersions are
quasi-compact and open immersions are locally of finite presentation. Finally
(vi) follows immediately from (iv) and (v).
\end{proof}


The Zariski topology induces a partial ordering on the underlying set of
points~\cite[2.1.1]{egaI_NE}. We let $x\leq x'$ if $x\in \overline{\{x'\}}$,
i.e., if $x$ is a specialization of $x'$, or equivalently if
$\overline{\{x\}}\subseteq\overline{\{x'\}}$. The $S$-topology is the topology
associated to this ordering. A subset is thus closed (resp.\ open) if and only
if it is stable under specialization (resp.\ generization). We denote by $S(E)$
the closure of $E$ in the $S$-topology. By $\overline{E}$ we will always mean
the closure of $E$ in the Zariski topology. A morphism of schemes
$\map{f}{X}{Y}$ is \emph{generizing} (resp.\ \emph{specializing})
if it is open (resp.\ closed) in the $S$-topology~\cite[\S3.9]{egaI_NE}.
An open (resp.\ closed) morphism of schemes is generizing (resp.\
specializing)~\cite[Prop.~3.9.3]{egaI_NE}.

\begin{remark}
For an affine scheme $\Spec(A)$ the partial ordering described above
corresponds to \emph{reverse} inclusion of prime ideals and a maximal point
corresponds to a minimal ideal. In commutative algebra, it is common to take
the ordering on the spectrum corresponding to inclusion of prime ideals, but
this is less natural from a geometric viewpoint.
\end{remark}

\begin{proposition}[{\cite[Thm.~7.3.1]{egaI_NE}}]
\label{P:pro-constr-and-closure}
Let $X$ be a scheme. If $E\subseteq X$ is an ind-constructible subset then
$x\in \interior(E)$ if and only if $\Spec(\sO_{X,x})\subseteq E$. Equivalently,
we have that the interior of $E$ in the Zariski topology coincides with the
interior of $E$ in the $S$-topology. If $F\subseteq X$ is a pro-constructible
subset then $\overline{F}=S(F)$.
\end{proposition}

\begin{corollary}\label{C:constr+S<=>Zariski}
Let $X$ be a scheme. A subset $E\subseteq X$ is open (resp.\ closed) in the
Zariski topology if and only if $E$ is open (resp.\ closed) in both the
constructible topology and the $S$-topology.
\end{corollary}
\begin{proof}
As a closed (resp.\ open) immersion is quasi-compact (resp.\ locally of finite
presentation),
it follows that
the constructible topology is finer than the Zariski topology. That the
$S$-topology is finer than the Zariski topology is obvious. This shows the
``only if'' part. The ``if'' part follows from
Proposition~\pref{P:pro-constr-and-closure}.
\end{proof}

A map of topological spaces $\map{f}{X}{Y}$ is \emph{submersive} or a
submersion if $f$ is surjective and $Y$ has the quotient topology, i.e.,
$E\subseteq Y$ is open (resp.\ closed) if and only if $f^{-1}(E)$ is open
(resp.\ closed). We say that a morphism of schemes $\map{f}{X}{Y}$ is
submersive if the underlying morphism of topological spaces is submersive. We
say that $f$ is \emph{universally submersive} if $\map{f'}{X\times_Y Y'}{Y'}$
is submersive for every morphism of schemes $Y'\to Y$.

The composition of two submersive morphisms is submersive and if the
composition $g\circ f$ of two morphisms is a submersive morphism then so is
$g$. It follows immediately from Corollary~\pref{C:constr+S<=>Zariski} that if
$f$ is submersive in both the constructible and the $S$-topology, then $f$ is
submersive in the Zariski topology.


\begin{proposition}\label{P:submersions-in-cons}
Let $\map{f}{X}{Y}$ be a surjective morphism of schemes. Then $f^{\cons}$ is
submersive in the following cases:
\begin{enumerate}
\item $f$ is quasi-compact.
\item $f$ is locally of finite presentation.
\item $f$ is open.
\end{enumerate}
\end{proposition}
\begin{proof}
If $f$ is quasi-compact (resp.\ locally of finite presentation) then
$f^{\cons}$ is closed (resp.\ open) by
Proposition~\pref{P:constructible-topology} and it follows that $f^{\cons}$ is
submersive.

Assume that $f$ is open. Taking an open covering, we can assume that $Y$ is
affine.
As $f$ is open there is then a quasi-compact open subset $U\subseteq X$ such
that $f|_U$ is surjective. As $f|_U$ is quasi-compact it follows by part (i)
that $f^{\cons}|_U$ is submersive. In particular, we have that $f^{\cons}$ is
submersive.
\end{proof}

\begin{proposition}\label{P:descent-of-top-props}
Let $\map{f}{X}{Y}$ and $\map{g}{Y'}{Y}$ be morphisms of schemes and
let $\map{f'}{X'}{Y'}$ be the pull-back of $f$ along $g$.
\begin{enumerate}
\item Assume that $g$ is submersive. If $f'$ is open
(resp.\ closed, resp.\ submersive) then so is $f$.
\item Assume that $g$ is universally submersive. Then $f$ has one of the
properties: universally open, universally closed, universally submersive,
separated; if and only if $f'$ has the same property.
\item Assume that $g^{\cons}$ is submersive. Then $f$ is quasi-compact if and
only if $f'$ is quasi-compact.
\end{enumerate}
\end{proposition}
\begin{proof}
(i) Assume that $f'$ is open (resp.\ closed) and let $Z\subseteq X$ be an open
(resp.\ closed) subset. Then $g^{-1}(f(Z))=f'(g'^{-1}(Z))$ is open
(resp.\ closed) and thus so is $f(Z)$ if $g$ is submersive. If $f'$ is
submersive then so is $g\circ f'=f\circ g'$ which shows that $f$ is
submersive.
The first three properties of (ii) follow easily from (i) and if $f$ is
separated then so is $f'$. If $f'$ is separated, then $\Delta_{X'/Y'}$ is
universally closed and it follows that $\Delta_{X/Y}$ is universally closed and
hence a closed immersion~\cite[Cor.~18.12.6]{egaIV}.

(iii) If $f$ is quasi-compact then $f'$ is quasi-compact. Assume that $f'$ is
quasi-compact and that $g^{\cons}$ is submersive. Then $f'^{\cons}$ is closed
by Proposition~\pref{P:constructible-topology} and it follows as in (i) that
$f^{\cons}$ is closed. Moreover, the fibers of $f$ are quasi-compact as the
fibers of $f'$ are quasi-compact. If $y\in Y$ then $(X_y)^{\cons}$ is
quasi-compact~\cite[Prop.~7.2.13 (i)]{egaI_NE} and so is the image
$(X^{\cons})_y$ of $(X_y)^{\cons}\to X^{\cons}$. Thus $f^{\cons}$ is proper
since it is closed with quasi-compact fibers, and it follows that $f$ is
quasi-compact by~\cite[Prop.~7.2.13 (v)]{egaI_NE}.
\end{proof}

\begin{remark}
Let us
indicate how to extend the results of this section from schemes to algebraic
spaces. Recall that associated to every algebraic space $X$ is an underlying
topological space $|X|$
and that a morphism $f$ of algebraic spaces induces a continuous map $|f|$ on
the underlying spaces~\cite[II.6]{knutson_alg_spaces}. By definition, a
morphism of algebraic spaces $\map{f}{X}{Y}$ is \emph{submersive} if $|f|$ is
submersive. If $U$ is a scheme and $\map{f}{U}{X}$ is \etale{} and surjective,
then $|f|$ is submersive. A morphism of algebraic spaces $\map{f}{X}{Y}$ is
\emph{universally submersive} if $\map{f'}{X\times_Y Y'}{Y'}$ is submersive for
every morphism $Y'\to Y$ of algebraic spaces. For $f$ to be universally
submersive it is sufficient that $f'$ is submersive for every (affine) scheme
$Y'$.

The constructible topology (resp.\ $S$-topology) on the set $|X|$ is
the quotient topology of the corresponding topology
on $|U|$ for an \etale{} presentation $U\to X$. This definition is readily
seen to be independent on the choice of presentation. 
The results \ref{P:char-of-pro-constr}--\ref{P:descent-of-top-props} then
follow by taking \etale{} presentations.

It is also possible to define the constructible topology and the $S$-topology
for a (quasi-separated) algebraic space intrinsically. In fact, the notions of
specializations, constructible, pro-constructible and ind-constructible sets
are meaningful for any topological space. To see that these two definitions
agree, it is enough to show that if $U$ is a scheme and $\map{f}{U}{X}$ is an
\etale{} presentation, then $f$ is submersive in both the constructible
topology and the $S$-topology. That $f$ is submersive in the $S$-topology
follows from~\cite[Cor.~5.7.1]{laumon}. That $f$ is submersive in the
constructible topology follows from Chevalley's
Theorem~\cite[Thm.~5.9.4]{laumon}
%
%
but its proof in \loccit{}\ uses~\cite[Cor.~5.9.2]{laumon} which appears
to have an incorrect proof as only locally closed subsets and not finite unions
of such are considered. We now give a different proof:

As the question is local, we can assume that $X$ is quasi-compact (and
quasi-separated).
Then $X$ has a finite stratification into locally closed constructible
subspaces $X_i$ such that the $X_i$'s are quasi-compact and quasi-separated
\emph{schemes}~\cite[Prop.~5.7.6]{raynaud-gruson}. The induced morphism
$\coprod_i X_i\to X$ is a universal homeomorphism in the constructible
topology and it follows that $f^{\cons}$ is submersive from the usual result
for schemes.
\end{remark}

\end{section}


\begin{section}{Subtrusive morphisms}\label{S:subtrusive}

In this section, we define and give examples of subtrusive morphisms. We then
give two valuative criteria and show that for noetherian schemes every
universally submersive morphism is universally subtrusive.

\begin{proposition}\label{P:S-subtrusive}
Let $\map{f}{X}{Y}$ be a morphism of schemes. The following are equivalent.
\begin{enumerate}
\item Every ordered pair $y\leq y'$ of points in $Y$ lifts to an
ordered pair of points $x\leq x'$ in $X$.
\label{PI:SS-pairs}
\item For every point $y\in Y$ we have that
$f(S(f^{-1}(y)))=S(y)$.
\item For every subset $Z\subseteq Y$ we have that
$f(S(f^{-1}(Z)))=S(Z)$.
\item For every pro-constructible subset $Z\subseteq Y$ we have that
$f\left(\overline{f^{-1}(Z)}\right)=\overline{Z}$.
\end{enumerate}
Under these equivalent conditions $f$ is submersive in the $S$-topology.
\end{proposition}
\begin{proof}
It is clear that (i)$\iff$(ii). As specialization commutes with unions, it is
clear that (ii)$\iff$(iii). By Proposition~\pref{P:pro-constr-and-closure} it
follows that (iii)$\implies$(iv). As every point of $Y$ is pro-constructible we
have that (iv)$\implies$(ii). Finally it is clear that $f$ is submersive in the
$S$\nobreakdash-topology when (iii) is satisfied.
\end{proof}

\begin{definition}\label{D:subtrusive}
We call a morphism of schemes \emph{$S$-subtrusive} if the equivalent
conditions of Proposition~\pref{P:S-subtrusive} are satisfied.
We say that a morphism is \emph{subtrusive} if it is $S$-subtrusive and
submersive in the constructible topology.
We say that $\map{f}{X}{Y}$ is \emph{universally subtrusive} if
$\map{f'}{X\times_Y Y'}{Y'}$ is subtrusive for every morphism $Y'\to Y$.
\end{definition}

\begin{remark}
Picavet only considers spectral spaces and spectral morphisms, i.e.,
topological spaces that are the spectra of affine rings and quasi-compact
morphisms of such spaces~\cite{hochster_prime_ideal_structure}.
In particular, surjective spectral morphisms are submersive in the constructible
topology. Taking this into account, Picavet's definition of subtrusive
morphisms~\cite[D\'ef.~2]{picavet_submersion} agrees with
Definition~\pref{D:subtrusive}. Instead of ``$S$-subtrusive'' Picavet uses
either ``strongly subtrusive in the $S$-topology'' or ``strongly submersive
of the first order in the $S$-topology''.

Note that the non-trivial results in this paper deal with
\emph{quasi-compact} subtrusive morphism. Nevertheless, we have chosen to give
the general definition of subtrusive morphisms as this clarifies the usage of
the constructible topology.
\end{remark}


Every subtrusive morphism is submersive by
Corollary~\pref{C:constr+S<=>Zariski}. It is furthermore clear that the
composition of two subtrusive morphisms is subtrusive and that if the
composition $g\circ f$ of two morphisms is subtrusive then so is $g$.

\begin{proposition}\label{P:locality-of-subtrusiveness}
Let $\map{f}{X}{Y}$ be a morphism of schemes.
\begin{enumerate}
\item $f$ is submersive (resp.\ subtrusive) if and only if $f_\red$ is
submersive (resp.\ subtrusive).
\item Let $Y=\bigcup Y_i$ be an open covering. Then $f$ is submersive
(resp.\ subtrusive) if and only if $f|_{Y_i}$ is submersive (resp.\ subtrusive)
for every $i$.
\item Let $W$ be a locally closed subscheme of $Y$. If $f$ is submersive
(resp.\ subtrusive) then so is $f|_W$.
\end{enumerate}
\end{proposition}
\begin{proof}
It is enough to verify the corresponding statements for: $f$ submersive,
$f^{\cons}$ submersive and $f$ $S$-subtrusive. This
follows easily from the topological definition of submersive,
Proposition~\pref{P:constructible-topology} and the characterization of
$S$-subtrusive given in~\ref{PI:SS-pairs} of Proposition~\pref{P:S-subtrusive}.
\end{proof}

\begin{remark}\label{R:subtrusive-examples}
Let $\map{f}{X}{Y}$ be a surjective morphism of schemes. Then $f$ is
universally subtrusive in the following cases:
\begin{enumerate}
\item[(1)] $f$ is universally specializing and $f^{\cons}$ is universally
submersive.
\item[(1a)] $f$ is proper.
\item[(1b)] $f$ is integral.
\item[(1c)] $f$ is essentially proper, i.e., universally specializing, separated
and locally of finite presentation (defined in~\cite[Rem.~18.10.20]{egaIV} for
$Y$ locally noetherian).
\item[(2)] $f$ is universally generizing and $f^{\cons}$ is universally
submersive.
\item[(2a)] $f$ is fpqc (faithfully flat and quasi-compact).
\item[(2b)] $f$ is fppf (faithfully flat and locally of finite presentation).
\item[(2c)] $f$ is universally open.
\end{enumerate}
Recall that quasi-compact morphisms, morphisms locally of finite presentation
and universally open morphisms are universally submersive in the constructible
topology, cf.~Proposition~\pref{P:submersions-in-cons}. Thus (1a)--(1c) are
special cases of (1) and as flat and open morphisms are generizing it is clear
that (2a)--(2c) are special cases of (2). That $f$ is universally subtrusive in
(1) and (2) is obvious.
\end{remark}

\begin{remark}\label{R:val-rings-and-flatness}
Let $V$ be a valuation ring. Then every finitely generated ideal in $V$ is
principal and a $V$-module is flat if and only if it is torsion
free~\cite[Ch.~VI, \S3, No.~6, Lem.~1]{bourbaki_alg_comm_5_6}.
In particular, if $B$ is a $V$-algebra and an integral domain, then $B$ is flat
if and only if $V\to B$ is injective.
\end{remark}

\begin{proposition}[{\cite[Part~II, Prop.~1.3.1]{raynaud-gruson},
\cite[Prop.~16]{picavet_submersion}}]
\label{P:subtrusive-over-valuation-ring}
Let $V$ be a valuation ring and $\map{f}{X}{\Spec(V)}$ a morphism of
schemes. The following are equivalent:
\begin{enumerate}
\item $f$ is universally subtrusive.\label{PI:st-vr:univ-subt}
\item $f$ is subtrusive.\label{PI:st-vr:subt}
\item $f$ is $S$-subtrusive.\label{PI:st-vr:s-subt}
\item The closure of the generic fiber $X\times_V \Spec(K)$ in $X$
surjects onto~$V$.\label{PI:st-vr:closure}
\item The pair $\im\leq (0)$ in $\Spec(V)$ lifts to $x\leq x'$
in $X$.\label{PI:st-vr:lifting}
\item There is a valuation ring $W$ and a morphism $\Spec(W)\to X$ such that
the composition $\Spec(W)\to X\to \Spec(V)$ is surjective.
\label{PI:st-vr:lifting-to-val-ring}
\item Any chain of points in $\Spec(V)$ lifts to a chain of points in
$X$.\label{PI:st-vr:chain-lifting}
\item There is a closed subscheme $Z\inj X$ such that $f|_Z$ is faithfully
flat.\label{PI:st-vr:flat}
\end{enumerate}
If $V$ is a discrete valuation ring, then the above conditions are equivalent
with the following:
\begin{enumerate}\setcounter{enumi}{8}
\item $f$ is submersive.\label{PI:st-vr:subm}
\end{enumerate}
\end{proposition}
\begin{proof}
It is clear that \ref{PI:st-vr:univ-subt}$\implies$\ref{PI:st-vr:subt}%
$\implies$\ref{PI:st-vr:s-subt}
$\implies$\ref{PI:st-vr:closure}$\implies$\ref{PI:st-vr:lifting} and that
\ref{PI:st-vr:lifting-to-val-ring}$\implies$\ref{PI:st-vr:chain-lifting}%
$\implies$\ref{PI:st-vr:lifting}. We will now show that~\ref{PI:st-vr:lifting}
implies~\ref{PI:st-vr:lifting-to-val-ring} and~\ref{PI:st-vr:flat}.
We let $Z=\overline{\{x'\}}$ which is
an integral closed subscheme of $X$ dominating $\Spec(V)$. We let $W$ be a
valuation ring dominating $\sO_{Z,x}$. Then both $\Spec(W)\to\Spec(V)$ and
$Z\to\Spec(V)$ are flat by Remark~\pref{R:val-rings-and-flatness}. As
the images in $\Spec(V)$ of $\Spec(W)$ and $Z$ contain the closed point,
we have that $\Spec(W)\to\Spec(V)$ and $Z\to\Spec(V)$ are surjective.

If~\ref{PI:st-vr:flat} is satisfied then we let $x\in Z$ be a point over the
closed point of $\Spec(V)$. The morphism $\Spec(\sO_{Z,x})\to Z\inj X\to
\Spec(V)$ is faithfully flat and quasi-compact and hence universally subtrusive
by case (2a) in Remark~\pref{R:subtrusive-examples}. In particular, we have
that $X\to\Spec(V)$ is universally subtrusive.

Finally~\ref{PI:st-vr:subt} always implies~\ref{PI:st-vr:subm} and if $V$ is a
discrete valuation ring, then~\ref{PI:st-vr:subm}
implies~\ref{PI:st-vr:closure}.
\end{proof}

In the proof of the following theorem we use
Corollary~\pref{C:univ-submersive-An}. This corollary is independent
of Theorem~\pref{T:submersive=subtrusive-if-noeth} as the results of
\S\S\ref{N:limit}--\ref{C:univ-submersive-An} only uses the basic properties
\S\S\ref{P:S-subtrusive}--\ref{P:locality-of-subtrusiveness} of subtrusive
morphisms.

\begin{theorem}[{\cite[Thm.~29, Thm.~37]{picavet_submersion}}]
\label{T:valuative_criterion_for_subtrusions}
\label{T:submersive=subtrusive-if-noeth}
Let $\map{f}{X}{Y}$ be a morphism such that $f^{\cons}$ is universally
submersive (e.g.\ $f$ quasi-compact).
\begin{enumerate}
\item $f$ is universally subtrusive if and only if, for any valuation ring $V$
and morphism $Y'\to Y$ with $Y'=\Spec(V)$, the pull-back $\map{f'}{X'}{Y'}$ is
subtrusive.\label{I:VC-subt}
\item $f$ is universally submersive if and only if, for any valuation ring $V$
and morphism $Y'\to Y$ with $Y'=\Spec(V)$, the pull-back $\map{f'}{X'}{Y'}$ is
submersive.\label{I:VC-subm}
\end{enumerate}
If $Y$ is locally noetherian then it is enough to consider discrete valuation
rings in~\ref{I:VC-subt} and~\ref{I:VC-subm}, and $f$ is universally subtrusive
if and only if $f$ is universally submersive.
\end{theorem}
\begin{proof}
The necessity of the conditions is clear. To prove the sufficiency
of~\ref{I:VC-subt}, take any base change $Y'\to Y$ and let $y\leq y'$ be an
ordered pair in $Y'$. We have to show that $y\leq y'$ can be lifted to
$X'$. There is a valuation ring $V$ and a morphism $\Spec(V)\to Y'$ such that
the pair $\im\leq (0)$ in $\Spec(V)$ lifts $y\leq y'$, cf.~\cite[Ch.~VI, \S1,
No.~2, Cor.]{bourbaki_alg_comm_5_6}.
As $X'\times_{Y'} \Spec(V)\to \Spec(V)$ is subtrusive by assumption we can then
lift $\im\leq (0)$ to an ordered pair in $X'\times_{Y'} \Spec(V)$ which after
projection onto $X'$ gives a lifting of $y\leq y'$.

To prove the sufficiency of~\ref{I:VC-subm}, assume that $f'$ is submersive
whenever $Y'$ is the spectrum of a valuation ring. It is then enough to show
that $f$ is submersive. Let $W\subseteq Y$ be a subset such that $f^{-1}(W)$ is
closed. Then as $f^\cons$ is submersive it follows that $W$ is
pro-constructible. We will now show that $W$ is closed under specialization.
Then $W$ is closed and it follows that $f$ is submersive. Let $y\leq y'$ be an
ordered pair in $Y$ with $y'\in W$ and choose a valuation ring $V$ with a
morphism $Y'=\Spec(V)\to Y$ such that the pair $\im\leq (0)$ in $Y'$ lifts
$y\leq y'$. Let $W'$ be the inverse image of $W$ along $Y'\to Y$. As
$\map{f'}{X'}{Y'}$ is submersive by assumption, we have that $W'$ is closed. As
$(0)\in W'$ it follows that $\im\in W'$ and thus $y\in W$.

When $Y$ is locally noetherian, it is enough to consider locally noetherian
base changes $Y'\to Y$ in~\ref{I:VC-subt} by
Corollary~\pref{C:univ-submersive-An}.
As every ordered pair in a noetherian scheme can be lifted to a discrete
valuation ring~\cite[Prop.~7.1.7]{egaII}, it is thus enough to consider
discrete valuation rings in~\ref{I:VC-subt}. Every universally subtrusive
morphism is universally submersive. To show the remaining statements, it is
thus enough to show that the valuative criteria in~\ref{I:VC-subt}
and~\ref{I:VC-subm} are equivalent over discrete valuation rings. This is the
equivalence of~\ref{PI:st-vr:subt} and~\ref{PI:st-vr:subm} in
Proposition~\pref{P:subtrusive-over-valuation-ring}.
\end{proof}

\begin{corollary}
Let $\map{f}{X}{Y}$ be a morphism such that $f^{\cons}$ is universally
submersive (e.g.\ $f$ quasi-compact). Then the following are equivalent:
\begin{enumerate}
\item $f$ is universally subtrusive.
\item For every valuation ring $V$ and diagram of solid arrows
$$\xymatrix{
\Spec(V')\ar@{.>}[r]\ar@{.>}[d] & X\ar[d]\\
\Spec(V)\ar[r] & Y
}$$
\end{enumerate}
there is a valuation ring $V'$ and morphisms such that the diagram
becomes commutative and such that the left vertical morphism is surjective.
\end{corollary}
\begin{proof}
Follows immediately from Proposition~\pref{P:subtrusive-over-valuation-ring}
and Theorem~\pref{T:valuative_criterion_for_subtrusions}.
\end{proof}

\begin{corollary}[{\cite[Prop.~3.7]{kollar_quotients},
\cite[Exp.~IX, Rem.~2.6]{sga1}}]\label{C:DVR-criterion-of-submersive}
Let $Y$ be \emph{locally noetherian} and let $\map{f}{X}{Y}$ be a morphism
\emph{locally of finite type}. Then the following are equivalent
\begin{enumerate}
\item $f$ is universally submersive.
\item $f$ is universally subtrusive.
\item For every discrete valuation ring $D$ and diagram of solid arrows
$$\xymatrix{
\Spec(D')\ar@{.>}[r]\ar@{.>}[d] & X\ar[d]\\
\Spec(D)\ar[r] & Y
}$$
\end{enumerate}
there is a discrete valuation ring $D'$ and morphisms making the diagram
commutative and such that the left vertical morphism is surjective.
\end{corollary}
\begin{proof}
Note that each of (i), (ii) and (iii) implies that $f$ is surjective. As $f$ is
locally of finite presentation, it is thus universally submersive in the
constructible topology by Proposition~\pref{P:submersions-in-cons}
under any of these conditions.

The equivalence of (i) and (ii) follows from
Theorem~\pref{T:submersive=subtrusive-if-noeth}. If (ii) is satisfied and $D$
is a discrete valuation ring with a morphism to $Y$ then $X\times_Y \Spec(D)\to
\Spec(D)$ is subtrusive. We can thus find an ordered pair $x\leq x'$ in
$X\times_Y \Spec(D)$ above $\im\leq (0)$ in $\Spec(D)$. As $f$ is locally of
finite type we have that $X$ is locally noetherian and we can find a discrete
valuation ring $D'$ with a morphism $\Spec(D')\to X\times_Y \Spec(D)$ with
image $\{x,x'\}$. This shows that (ii) implies (iii).

Finally, assume that we have a diagram as in (iii). Then the morphism
$\Spec(D')\to \Spec(D)$ is submersive and hence so is $X\times_Y \Spec(D)\to
\Spec(D)$.
Thus (iii) implies (ii) by
Theorem~\pref{T:valuative_criterion_for_subtrusions}.
\end{proof}

For completeness we mention the following result which is an immediate
consequence of Proposition~\pref{P:subtrusive-over-valuation-ring} and
a result of Kang and Oh~\cite{kang-oh_lifting-chains}.

%
%

\begin{proposition}[{\cite[Thm.~3.26]{picavet_GGD}}]
\label{P:lifting-of-chains}
Let $\map{f}{X}{Y}$ be a universally subtrusive morphism and let $\{y_\alpha\}$
be a chain of points in $Y$. Assume that $\{y_\alpha\}$ has a lower bound in
$Y$ or equivalently that $\{y_\alpha\}$ is contained in an affine open subset
of $Y$.
There is then a chain $\{x_\alpha\}$ of points in $X$ which lifts the
chain in $Y$, i.e., such that $f(x_\alpha)=y_\alpha$ for every $\alpha$.
\end{proposition}
\begin{proof}
We can assume that $Y$ is affine. The closure in the Zariski topology of the
chain $\{y_\alpha\}$ is irreducible so we can also assume that $Y$ is integral.
By~\cite{kang-oh_lifting-chains} there exists a valuation ring $V$, a morphism
$\Spec(V)\to Y$ and a lifting of the chain $\{y_\alpha\}$ to a chain
$\{v_\alpha\}$ in $\Spec(V)$. By
Proposition~\pref{P:subtrusive-over-valuation-ring} there is then a lifting of
the chain $\{v_\alpha\}$ to a chain in $X\times_Y \Spec(V)$. The projection of
this chain onto $X$ gives a lifting of the chain $\{y_\alpha\}$.
\end{proof}


The results of this section, except possibly
Proposition~\pref{P:lifting-of-chains}, readily generalize to algebraic spaces.
Proposition~\pref{P:lifting-of-chains} is at least valid for finite chains as
such lift over \etale{} surjective morphisms.
\end{section}


\begin{section}{Structure theorem for finitely presented subtrusions}
\label{S:structure-theorems}

In this section, we give a structure theorem for finitely presented universally
subtrusive morphisms. This result is an extension
of~\cite[Thm.~3.1.9]{voevodsky_homology} to the non-noetherian case. If
$\map{f}{X}{Y}$ is universally subtrusive of finite presentation, then we will
show the existence of a refinement $X'\to X\to Y$ of $f$ such that there is a
factorization $X'\to Y'\to Y$ where the first morphism is an open covering and
the second is a surjective, proper and finitely presented morphism. If in
addition $f$ is quasi-finite then there is a similar refinement with a
factorization in which the second morphism is finite.

\begin{notation}
Let $X$ be a scheme. We denote by $X_\genpts$
the set of maximal points of $X$, i.e., the generic points of the irreducible
components of $X$.
\end{notation}

The main tools we will use are the ``flatification by blowup''-result of
Raynaud and Gruson~\cite{raynaud-gruson} and the following lemma:

\begin{lemma}\label{L:surjectivity-of-subtrusions}
Let $\map{f}{X}{Y}$ be a morphism of schemes and let $U\subseteq Y$ be
any subset containing $Y_\genpts$. Let $V=\overline{f^{-1}(U)}$ be the closure
in the Zariski topology. If $f$ is $S$-subtrusive then $f|_V$ is surjective.
\end{lemma}
\begin{proof}
Let $y\in Y$ and choose a generization $y'\in U$. As $f$ is $S$-subtrusive, the
pair $y\leq y'$ lifts to a pair $x\leq x'$. We have that
$x\in S(f^{-1}(U))\subseteq \overline{f^{-1}(U)}$.
\end{proof}

\begin{definition}
We say that a morphism $\map{p}{\widetilde{S}}{S}$ is a \emph{blow-up}, if
there is a closed subscheme $Z\inj S$ given by a finitely generated ideal
sheaf, such that $\widetilde{S}$ is the blow-up of $S$ in $Z$. Then $p$ is
proper and an isomorphism over the retrocompact open subset $U=S\setminus
Z$. Let $\map{f}{X}{S}$ be another morphism. The \emph{strict transform}
$\widetilde{X}$ of $X$ under $p$ is the schematic closure of $f^{-1}(U)$ in
$X\times_S \widetilde{S}$. The strict transform $\widetilde{f}$ of $f$ under
$p$ is the composition $\map{\widetilde{f}}{\widetilde{X}\inj X\times_S
\widetilde{S}}{\widetilde{S}}$.
\end{definition}

\begin{remark}\label{R:surj-of-strict-tfm-of-subtrusion}
Let $\map{f}{X}{Y}$ be universally subtrusive and let
$\map{p}{\widetilde{Y}}{Y}$ be a blow-up. By
Lemma~\pref{L:surjectivity-of-subtrusions} it follows that the strict
transform $\widetilde{f}$ of $f$ is surjective.
\end{remark}

To begin with, we will need the technical condition that $Y_\genpts$ is
quasi-compact. Note that if $Y$ is quasi-separated then $Y_\genpts$ is always
Hausdorff, cf.~\cite[Ch.~I, Lem.~2.8]{lazard_autour}. Rings $A$ such that
$\Min(A)=\Spec(A)_\genpts$ is compact are studied
in~\cite[Ch.~II]{olivier_sem_samuel}.


\begin{lemma}\label{L:gen-flatness}
Let $Y$ be a reduced quasi-compact and quasi-separated scheme such that
$Y_\genpts$ is quasi-compact, e.g., $Y$ is reduced and noetherian. Let
$\map{f}{X}{Y}$ be a finitely presented morphism. There is then an open dense
quasi-compact subset $U\subseteq Y$ such that $f$ is flat over $U$.
\end{lemma}
\begin{proof}
As $Y$ is reduced $f$ is flat over $Y_\genpts$ and hence flat over an open
subset $V\subseteq Y$ containing $Y_\genpts$, cf.~\cite[Cor.~11.3.2]{egaIV}.
As $Y_\genpts$ is quasi-compact, there is an open quasi-compact subset
$U\subseteq V$ containing $Y_\genpts$.
\end{proof}

\begin{proposition}\label{P:ST-qc}
Let $Y$ be a \emph{reduced} quasi-compact and quasi-separated scheme such that
$Y_\genpts$ is quasi-compact, e.g., $Y$ noetherian. Let $\map{f}{X}{Y}$ be a
universally subtrusive morphism of finite presentation. Then there is a
surjective blow-up $\widetilde{Y}\to Y$ (of finite type) such that the
strict transform $\map{\widetilde{f}}{\widetilde{X}}{\widetilde{Y}}$ is
faithfully flat of finite presentation.
\end{proposition}
\begin{proof}
By Lemma~\pref{L:gen-flatness} there is an open quasi-compact dense subset
$U$ over which $f$ is flat. By~\cite[Thm.~5.2.2]{raynaud-gruson}, there is a
blow-up $\map{p}{\widetilde{Y}}{Y}$ such that $p$ is an isomorphism over $U$
and such that the strict transform $\widetilde{f}$ is flat and finitely
presented. By Remark~\pref{R:surj-of-strict-tfm-of-subtrusion} the morphism
$\widetilde{f}$ is surjective.
\end{proof}

\begin{proposition}\label{P:ST-aff+q-compact}
Let $Y$ be affine and such that $Y_\genpts$ is quasi-compact, e.g., $Y$
irreducible. Let $\map{f}{X}{Y}$ be a universally subtrusive morphism of finite
presentation. Then there is a refinement $\map{f'}{X'}{Y}$ of $f$ and a
factorization of $f'$ into a faithfully flat morphism $X'\to Y'$ of finite
presentation followed by a proper surjective morphism $Y'\to Y$ of finite
presentation. If in addition $f$ is universally open, then we may choose
$f'$ such that $X'\to X\times_Y Y'$ is a nil-immersion.
\end{proposition}
\begin{proof}
Write $Y$ as an inverse limit of noetherian affine schemes $Y_\lambda$. By
Lemma~\pref{L:gen-flatness} there is an open quasi-compact dense subset $U$
such that $f$ is flat over $U_\red$. By~\cite[Cor.~8.2.11]{egaIV} there is an
index $\lambda$ and an open subset $U_\lambda\subseteq Y_\lambda$ such that
$U=U_\lambda\times_{Y_\lambda} Y$. Increasing $\lambda$, we may then assume
that there is a finitely presented morphism
$\map{f_\lambda}{X_\lambda}{Y_\lambda}$ such that $X\iso
X_\lambda\times_{Y_\lambda} Y$ and such that $f_\lambda$ is flat over
$(U_\lambda)_\red$~\cite[Thm.~11.2.6]{egaIV}.
By~\cite[Thm.~5.2.2]{raynaud-gruson}, there is a blow-up
$\map{p}{\widetilde{Y_\lambda}}{(Y_\lambda)_\red}$ such that $p$ is an
isomorphism over $(U_\lambda)_\red$ and such that the strict transform
$\map{\widetilde{f'_\lambda}}{\widetilde{X_\lambda}}{\widetilde{Y_\lambda}}$ of
$\map{f'_\lambda}{X_\lambda\times_{Y_\lambda}
(Y_\lambda)_\red}{(Y_\lambda)_\red}$ along $p$ is flat. Note that
$\widetilde{f'_\lambda}$ is not necessarily surjective.

Let $\map{f_1}{X'}{Y'}$ (resp.\ $\map{f_2}{Y'}{Y}$) be the pull-back of
$\map{\widetilde{f'_\lambda}}{\widetilde{X_\lambda}}{\widetilde{Y_\lambda}}$
(resp.\ $\widetilde{Y_\lambda}\to (Y_\lambda)_\red\inj Y_\lambda$) along $Y\to
Y_\lambda$. Then $f_1$ is flat and of finite presentation, and $f_2$ is proper,
surjective and of finite presentation. We will now show that $f_1$ is
surjective. Note that $f_2$ is an isomorphism over $U_\red$ and that $X'\inj
X\times_Y Y'$ is an isomorphism over $X\times_Y \left(U_\red\right)$.

Let $\widetilde{X}$ be the closure of $X\times_Y \left(U_\red\right)$ in
$X\times_Y Y'$. We then have a canonical factorization $\widetilde{X}\inj
X'\inj X\times_Y Y'$. As $f$ is universally subtrusive $\widetilde{X}\to Y'$
is surjective by Lemma~\pref{L:surjectivity-of-subtrusions}. Thus $X'\to Y'$ is
surjective. If in addition $f$ is universally open, then $\widetilde{X}\inj
X\times_Y Y'$ is a nil-immersion and it follows that $X'\inj X\times_Y Y'$ is a
nil-immersion.
\end{proof}

To treat
the case where $Y_\genpts$ is not compact, we use
the \emph{total integral closure}.

\begin{definition}\label{D:TIC}
A scheme $X$ is said to be \emph{totally integrally closed} or TIC if:
\begin{enumerate}
\item $X$ is reduced.
\item For every $x\in X$, the closed subscheme $\overline{\{x\}}$ is
normal and has an algebraically closed field of fractions.
\item The underlying topological space of $X$ is
  extremal~\cite[\S2]{hochster_TIC}.
\end{enumerate}
\end{definition}

\begin{properties}\label{X:TIC}
We briefly list the basic properties of TIC~schemes.
\begin{enumerate}
\item $X$ is TIC if and only if $X$ is TIC on an open covering.
\item An affine TIC~scheme is the spectrum of a totally integrally closed
ring~\cite[Thm.~1]{hochster_TIC}.
\item If $X$ is TIC, quasi-compact and quasi-separated then $X_\genpts$ is
compact~\cite[Prop.~5]{hochster_TIC}.
%
\item If $X$ is TIC then for every $x\in X$, the local ring $\sO_{X,x}$ is
a strictly henselian normal domain~\cite[Prop.~7]{hochster_TIC},
\cite[Prop.~1.4]{artin_Hensel-join}.
%
%
%
\item If $\map{f}{X'}{X}$ is an affine morphism and $X'$ is TIC then
the integral closure of $X$ relative to $X'$ is TIC.\label{XTIC:IC-in-TIC}
\item Every reduced ring $A$ has an injective and integral
homomorphism into a totally integrally closed ring $\TIC(A)$,
cf.~\cite[p.~769]{hochster_TIC}. If $X=\Spec(A)$ then we denote the
corresponding TIC~scheme with $\TIC(X)=\Spec(\TIC(A))$. For an arbitrary affine
scheme $X$ we let $\TIC(X)=\TIC(X_\red)$.
\item If $X$ has a finite number of irreducible components, e.g., if $X$ is
noetherian, then there is a surjective and integral morphism $\TIC(X)\to X$
such that $\TIC(X)$ is totally integrally closed. Concretely, if
$x_1,x_2,\dots,x_n$ are the generic points of $X$ then $\TIC(X)$ is the
integral closure of $X_\red$ in $\Spec\bigl(\prod_i \overline{k(x_i)}\bigr)$. This is the
\emph{absolute integral closure} of $X$ introduced by
Artin~\cite[\S1]{artin_Hensel-join}.
\item Every monic polynomial with coefficients in a TIC~ring factors
completely into monic linear factors~\cite[p.~769]{hochster_TIC}.
\label{XTIC:monic}
\item If $X$ is an affine TIC~scheme and $Z\to X$ is a finite morphism of
schemes then there is a finite and finitely presented surjective morphism
$Z'\to Z$ such that $Z'$ is a disjoint union of closed subschemes
$Z_i\inj X$. This follows from~\ref{XTIC:monic}.
\label{XTIC:closed-cover}
\end{enumerate}
Note that, as with the algebraic closure of a field, $\TIC(A)$ is only unique
up to non-unique isomorphism and thus this construction does not immediately
extend to arbitrary schemes. It is possible to show that if $X$ is a
quasi-compact and quasi-separated scheme, then there is a TIC~scheme $X'$
together with a surjective integral morphism $X'\to X$. However, this
construction is slightly awkward and does not yield a unique $X'$.
\end{properties}

\begin{theorem}\label{T:ST-aff}
Let $Y$ be an affine or noetherian scheme. Let $\map{f}{X}{Y}$ be a universally
subtrusive morphism of finite presentation. Then there is a refinement
$\map{f'}{X'}{Y}$ of $f$ and a factorization of $f'$ into a faithfully flat
morphism $X'\to Y'$ of finite presentation followed by a proper surjective
morphism $Y'\to Y$ of finite presentation. If in addition $f$ is universally
open, then we may choose $f'$ such that $X'\to X\times_Y Y'$ is a
nil-immersion.
\end{theorem}
\begin{proof}
If $Y$ is noetherian, the theorem follows from Proposition~\pref{P:ST-qc}. If
$Y$ is affine, then we have a surjective integral morphism $\TIC(Y)\to Y$ from
a TIC scheme. As $\TIC(Y)_\genpts$ is quasi-compact, we can by
Proposition~\pref{P:ST-aff+q-compact} find a refinement $X''\to X\times_Y
\TIC(Y)\to\TIC(Y)$ such that there is a factorization $X''\to Y''\to \TIC(Y)$
where the first morphism is faithfully flat of finite presentation and the
second is proper, surjective and finitely presented. If $f$ is universally
open, we may also assume that $X''\to X\times_Y Y''$ is a nil-immersion.

As the integral morphism $\TIC(Y)\to Y$ is the inverse limit of finite and
finitely presented $Y$-schemes $Y_\lambda$~\cite[Lem.~11.5.5.1]{egaIV}, it
follows that there is an index $\lambda$ and morphisms $X''_\lambda\to
Y''_\lambda\to Y_\lambda$, $X''_\lambda\to X$ with the same properties as
$X''\to Y''\to \TIC(Y)$~\cite[Thm.~8.10.5 (xii) and Thm.~11.2.6]{egaIV}. If in
addition $f$ is universally open, it follows from~\cite[Thm.~8.10.5 (ii) and
(vi)]{egaIV} that after increasing $\lambda$, we can assume that
$X''_\lambda\to X\times_Y Y''_\lambda$ is a nil-immersion. Putting
$X'=X''_\lambda$ and $Y'=Y''_\lambda$ gives a refinement with the required
factorization.
\end{proof}

\begin{theorem}\label{T:STf-noeth/aff}
Let $Y$ be an affine or noetherian scheme. Let $\map{f}{X}{Y}$ be a
\emph{quasi-finite} universally subtrusive morphism of finite
presentation. Then there is a refinement $X'\to Y$ of $f$ which is the
composition of an open covering $X'\to Y'$ of finite presentation and a finite
surjective morphism $Y'\to Y$ of finite presentation.
\end{theorem}
\begin{proof}
Replacing $X$ with an open covering, we can assume that $f$ is separated. By
Zariski's Main Theorem~\cite[Thm.~8.12.6]{egaIV} there is then a factorization
$X\to Y'\to Y$ where $\map{f_1}{X}{Y'}$ is an open immersion and
$\map{f_2}{Y'}{Y}$ is finite. If $Y$ is noetherian then $f_2$ is of finite
presentation. If $Y$ is affine then by~\cite[Rem.~8.12.7]{egaIV} we can find a
factorization such that $f_2$ is of finite presentation.

If we can obtain a refinement $X'\to \TIC(Y)$ of $X\times_Y \TIC(Y)\to \TIC(Y)$
with a factorization of the specified form, then by a limit argument there is a
similar refinement $X'_\lambda\to Y_\lambda$ of $X\times_Y Y_\lambda\to
Y_\lambda$ for some finitely presented finite morphism $Y_\lambda\to Y$. The
refinement $X'_\lambda\to Y_\lambda\to Y$ of $X\to Y$ then has a factorization
of the requested form. We can thus assume that $Y=\TIC(Y)$ is totally
integrally closed.

We will now show that $\map{f=f_2\circ f_1}{X}{Y'\to Y}$ has a refinement
$X'\to Y$ which is an open covering. To show this, we can replace $Y$ with an
open covering and assume that $Y$ is affine. Now as $Y$ is totally integrally
closed and affine, there is a finite and finitely presented surjective morphism
$Y''\to Y'$ such that $Y''$ is a finite disjoint union of closed subschemes
$Y_i\inj Y$, cf.\ Properties~\pref{X:TIC}\ \ref{XTIC:closed-cover}. Let
$X_i=X\times_{Y'} Y_i$. Then $X_i\to Y$ is the composition of an open
quasi-compact immersion $X_i\to Y_i$ and a closed immersion $Y_i\to Y$ of
finite presentation. We can replace $X$ with $\coprod_i X_i$ and $Y'$ with
$\coprod_i Y_i$.

Let $X'=\overline{f^{-1}(Y_\genpts)}\inj X$ with the reduced structure.
Then $X'\to Y$ is surjective by Lemma~\pref{L:surjectivity-of-subtrusions}.
We will now show that $X'=\coprod_i \interior(X_i)$ so that $X'\to Y$ is an
open covering. The key observation is that the immersion $X_i\to
Y_i\to Y$ is of finite presentation and hence \emph{constructible}.
Since every local ring of $Y$ is irreducible, it thus follows from
Proposition~\pref{P:pro-constr-and-closure} that the interior
of $X_i$ coincides with the closure of $Y_\genpts\cap X_i$ in $X_i$.
\end{proof}

The following theorem is~\cite[Thm.~3.1.9]{voevodsky_homology} except that
we do not require that $Y$ is an excellent noetherian scheme:

\begin{theorem}\label{T:ST-voevodsky}
Let $Y$ be an affine or noetherian scheme. Let $\map{f}{X}{Y}$ be a
universally subtrusive morphism of finite presentation. Then there is a
refinement $X'\to Y$ of $f$ which factors as a quasi-compact open covering
$X'\to Y'$ followed by a proper surjective morphism $Y'\to Y$ of finite
presentation.
\end{theorem}
\begin{proof}
By Theorem~\pref{T:ST-aff} we have a refinement
$X'\to Y$ of $f$ together with a factorization $X'\to Y'\to Y$ where $X'\to Y'$
is fppf and $Y'\to Y$ is proper. Taking a
quasi-section~\cite[Cor.~17.16.2]{egaIV} we can in addition assume that $X'\to
Y'$ is quasi-finite. If $Y$ is not noetherian but affine, we can write $Y$ as a
limit of noetherian schemes and consequently we can assume that $Y$ and $Y'$
are noetherian.
By Theorem~\pref{T:STf-noeth/aff} we can now refine $X'\to Y'$ into an open
covering followed by a finite morphism.
\end{proof}

\begin{remark}
Using the limit methods of Thomason and
Trobaugh~\cite[App.~C]{thomason-trobaugh}, we can replace the condition that
$Y$ is affine or noetherian in
Theorems~\pref{T:ST-aff}--\pref{T:ST-voevodsky}
with the condition that $Y$ is quasi-compact and quasi-separated.
\end{remark}

\begin{remark}\label{R:st-non-fp}
If $Y$ is quasi-compact and quasi-separated and $\map{f}{X}{Y}$ is a
quasi-separated morphism of finite type, then there is a finitely presented
morphism $\map{f'}{X'}{Y}$ and a closed immersion $X\inj X'$ of
$Y$-schemes. This follows from similar limit methods as
in~\cite[App.~C]{thomason-trobaugh}, cf.\ \cite[Thm.~4.3]{conrad_nagata}.
Using this fact, analogues of Theorems~\pref{T:ST-aff}--\pref{T:ST-voevodsky}
for universally subtrusive morphisms of \emph{finite type} can be proved, at
least if the base scheme has a finite number of irreducible components. In
these analogues, the flat and open coverings are of finite presentation but the
proper and the finite morphisms need not be. For example,
Proposition~\pref{P:ST-qc} for $\map{f}{X}{Y}$ of finite type and $Y$ with a
finite number of components follows from~\cite[Thm.~3.4.6]{raynaud-gruson}.
\end{remark}

\end{section}


\begin{section}{Descent of locally closed subsets}\label{S:loc-closed}
Recall that a subset $E\subseteq X$ is locally closed if every point $x\in E$
admits an open neighborhood $U$ such that $E\cap U$ is closed in
$U$. Equivalently, $E$ is the intersection of an open subset and a closed
subset.
Recall that a locally closed subset $E\subseteq X$ is retrocompact if and only
if $E$ is pro-constructible and if and only if $E\to X$ is quasi-compact.

Let $\map{f}{S'}{S}$ be a faithfully flat and quasi-compact morphism of schemes
and let $E\subseteq S$ be a subset. Then $E$ is locally closed and
retrocompact, if and only if $f^{-1}(E)\subseteq S'$ is locally closed and
retrocompact~\cite[Prop.~7.3.7]{egaI_NE}.
In this section, we give generalizations of this result for
universally subtrusive morphisms. The proof of Theorem~\pref{T:loc-closed-1}
only requires the results of \S\pref{S:subtrusive} whereas
Theorem~\pref{T:loc-closed-2} depends upon a structure theorem in
\S\pref{S:structure-theorems}.

\begin{theorem}\label{T:loc-closed-1}
Let $\map{f}{S'}{S}$ be a universally subtrusive morphism of schemes and let
$E\subseteq S$ be a subset. Then $E$ is locally closed and constructible if and
only if $f^{-1}(E)$ is locally closed and constructible. If
$f$ is also quasi-compact,
then $E$ is locally closed and retrocompact if and only if $f^{-1}(E)$ is
locally closed and retrocompact.
\end{theorem}
\begin{proof}
If $E$ is locally closed and constructible (resp.\ retrocompact) then so is
$f^{-1}(E)$. Assume that $E'=f^{-1}(E)$ is locally closed and constructible
(resp.\ retrocompact). Then $E$ is constructible (resp.\ pro-constructible) since
$f$ is submersive in the constructible topology, and we have that
$\overline{E}=S(E)$ by Proposition~\pref{P:pro-constr-and-closure}. The theorem
follows if we show that $Z=\overline{E}\setminus E=S(E)\setminus E$ is
closed. By Corollary~\pref{C:constr+S<=>Zariski} it is enough to show that $Z$
is pro-constructible and stable under specialization.

If $f$ is quasi-compact, then
$\map{f|_{\overline{E'}}}{\overline{E'}}{\overline{E}}$ is quasi-compact and
surjective since $f$ is $S$-subtrusive. In particular, we have that
$f|_{\overline{E'}}$ is submersive in the constructible topology. It follows
that $E$ is ind-constructible in $\overline{E}$ since $f^{-1}(E)=E'$ is open in
$\overline{E'}$. Thus, in both cases $E$ is constructible as a subset
of $\overline{E}$ and so is its complement $Z$.

Let $z\in Z$ and let $s\in S$ be a specialization of $Z$. Then there exists a
generization $e\in E$ of $z$ and we obtain the ordered triple $s\leq z\leq
e$ in $\overline{E}$. As $f$ is \emph{universally} subtrusive, there exists by
Proposition~\pref{P:lifting-of-chains} a lifting $s'\leq z'\leq e'$ of this
chain to $S'$ where $e'\in E'$ and $z'\notin E'$. As $E'$ is locally closed it
follows that $s'\notin E'$ and hence $s\in Z=\overline{E}\setminus E$.
\end{proof}

\begin{theorem}\label{T:loc-closed-2}
Let $\map{f}{S'}{S}$ be a morphism of algebraic spaces which is either
\begin{enumerate}
\item open and surjective,
\item closed and surjective,
\item universally subtrusive of finite presentation.
\end{enumerate}
Then a subset $E\subseteq S$ is locally closed if and only if
$f^{-1}(E)\subseteq S'$ is locally closed.
\end{theorem}
\begin{proof}
The condition is clearly necessary and the sufficiency when $f$ is as in (i)
or (ii) is an easy exercise left to the reader.
%
%
Let $f$ be as in (iii)
and assume that
$f^{-1}(E)$ is locally closed.
According to (i) the question is local in the \etale{} topology so we
can assume that $S$ and $S'$ are affine schemes.
By Theorem~\pref{T:ST-aff} there is a refinement $S''\to S$ of $f$ which
factors as an open surjective morphism followed by a closed surjective
morphism. It follows that $E$ is locally closed from
the cases (i) and (ii).
\end{proof}


The following example shows that neither theorem is true if we replace
universally subtrusive with universally submersive.

\begin{example}\label{E:submersive-non-subtrusive}
Let $S$ be the spectrum of a valuation ring $V$ of dimension two. Then
$\Spec(S)=\{x_0\leq x_1\leq x_2\}$. Let $s,t\in V$ be elements such that
$\Spec(V/s)=\{x_0,x_1\}$ and $\Spec(V_t)=\{x_1,x_2\}$. Let $S'=\Spec(V/s\times
V_t)$ with the natural morphism $\map{f}{S'}{S}$. Then $f$ is a universally
submersive morphism of finite presentation, cf.\
\cite[Cor.~33]{picavet_submersion}. Let $E=\{x_0,x_2\}\subset S$ be the subset
consisting of the minimal and the maximal point. Then $E$ is not locally closed
but $f^{-1}(E)$ is locally closed and constructible.
\end{example}

\end{section}


\begin{section}{Effective descent of \etale{} morphisms}\label{S:descent}

In this section, we will show that quasi-compact universally subtrusive
morphisms are morphisms of effective descent for the fibered category of
quasi-compact and separated \etale{} schemes. We will also show that this holds
for the fibered category of quasi-compact, but not necessarily separated,
\etale{} \emph{algebraic spaces}.

There is no need to include algebraic spaces when considering separated
\etale{} morphisms as any separated locally quasi-finite morphism of algebraic
spaces is representable by schemes. On the other hand, starting with a
non-separated \etale{} scheme equipped with a descent datum, this can descend
to an algebraic space which is not a scheme. We therefore need to extend the
basic results about \etale{} morphisms to algebraic spaces and this is done in
Appendix~\ref{S:etale-henselian}. The methods and results of this section are
similar to and extend those of~\cite[Exp.~IX]{sga1}.

\begin{notation}
Let $\Sch$ be the category of quasi-separated schemes. Let $\catE$ be the
following fibered category over $\Sch$: The objects of $\catE$ are \etale{}
morphisms $X\to S$ where $X$ is an algebraic space. The morphisms of $\catE$
are commutative squares $(X',S')\to (X,S)$. The structure functor
$\catE\to\Sch$ is the forgetful functor taking an object $X\to S$ to its target
$S$ and a morphism $(X',S')\to (X,S)$ to the morphism $S'\to S$.
We will also consider the following fibered full subcategories of $\catE$ where
the objects are:
\begin{align*}
\catE_{\sep} &= \{ \text{\etale{} and separated morphisms} \} \\
\catE_{\qc} &= \{ \text{\etale{} and quasi-compact morphisms} \} \\
\catE_{\sep,\qc} &= \{ \text{\etale{}, separated and
quasi-compact morphisms} \} \\
\catE_{\finite} &= \{ \text{\etale{} and finite morphisms} \}.
\end{align*}
It follows from Proposition~\pref{P:etale+sep-is-repr} that the objects of
$\catE_{\finite}\subseteq \catE_{\sep,\qc}\subseteq \catE_{\sep}$ are morphisms
of \emph{schemes}.
\end{notation}

\begin{remark}
We have chosen to use $\Sch$ as the base category for convenience. We could
instead have used the category of affine schemes or the category of algebraic
spaces and all results would have remained valid as can be seen from
Proposition~\pref{P:fppf-descent}.
Note that as algebraic spaces are assumed to be quasi-separated, the objects of
$\catE_{\qc}$ are of finite presentation. In particular, the category
$\catE_{\qc/S}$ is equivalent to the category of constructible sheaves on
$S$, cf.\ proof of Proposition~\pref{P:hensel-equiv}.
\end{remark}

\begin{proposition}[{\cite[Exp.~IX, Cor.~3.3]{sga1}}]
\label{P:univ-sub-is-m.o.d.}
Let $\map{f}{S'}{S}$ be a universally submersive morphism of schemes. Then $f$
is a morphism of $\catE$-descent. This means that for \etale{} morphisms $X\to
S$ and $Y\to S$ the sequence
$$\xymatrix{
\Hom_S(X,Y)\ar[r] & \Hom_{S'}(X',Y')\ar@<.5ex>[r]\ar@<-.5ex>[r]
 & \Hom_{S''}(X'',Y'')}$$
is exact, where $X'$ and $Y'$ are the pull-backs of $X$ and $Y$ along $S'\to
S$, and $X''$ and $Y''$ are the pull-backs of $X$ and $Y$ along $S''=S'\times_S
S'\to S$.
\end{proposition}
\begin{proof}
Follows easily from Corollary~\pref{C:WWW}.
\end{proof}

\begin{proposition}
Let $\map{f}{S'}{S}$ be a universally \emph{subtrusive} morphism of schemes.
Let $X\to S$ be an \etale{} morphism. If $X\times_S S'$ has one of the
properties: universally closed, separated, quasi-compact; then so has
$X\to S$. In particular, if $X\times_S S'\to S'$ lies in one of the categories:
$\catE_{\sep}$, $\catE_{\qc}$, $\catE_{\sep,\qc}$, $\catE_{\finite}$; then so
does $X\to S$.
\end{proposition}
\begin{proof}
This follows immediately from Proposition~\pref{P:descent-of-top-props}. For
the last statement, recall that the \etale{} morphism $X\to S$ is finite if and
only if it is separated, quasi-compact and universally
closed~\cite[Thm.~8.11.1]{egaIV}.
\end{proof}

\begin{xpar}[Descent data]
Let $S'\to S$ be any morphism and let $S''=S'\times_S S'$ and $S'''=S'\times_S
S'\times_S S'$. Let $X\to S$ be an \etale{} morphism, $X'=X\times_S S'$ and
$X''=X\times_S S''$. Then $X''$ is canonically $S''$-isomorphic with
$\pi_1^{*}X'$ and $\pi_2^{*}X'$ where $\map{\pi_1,\pi_2}{S''}{S'}$ are the two
projections. In particular we have an $S''$-isomorphism
$\map{\varphi}{\pi_1^{*}X'}{\pi_2^{*}X'}$ satisfying the \emph{cocycle
condition}, i.e., if $\map{\pi_{ij}}{S'''}{S''}$ denotes the projection on the
$i$\textsuperscript{th} and $j$\textsuperscript{th} factors then
\newcommand{\can}{\mathrm{can}}
$$\xymatrix@C=1mm@R=8.6mm{%
& \pi_{12}^{*}\pi_2^{*}X'\ar[rrrr]^{\can}_{\iso}
&&&& \pi_{23}^{*}\pi_1^{*}X'\ar[dr]^{\pi_{23}^*(\varphi)} & \\
\pi_{12}^{*}\pi_1^{*}X'\ar[ur]^{\pi_{12}^*(\varphi)}
&&&\circ&&& \pi_{23}^{*}\pi_2^{*}X'\ar[dl]_{\can}^{\iso} \\
& \pi_{31}^{*}\pi_2^{*}X'\ar[ul]_{\can}^{\iso}
&&&& \pi_{31}^{*}\pi_1^{*}X'\ar[llll]_{\pi_{31}^*(\varphi)} &
}$$
commutes.
Conversely, given an \etale{} morphism $X'\to S'$ we say that an
$S''$-isomorphism $\map{\varphi}{\pi_1^{*}X'}{\pi_2^{*}X'}$ satisfying the
cocycle condition is a \emph{descent datum} for $X'\to
S'$. We say that $(X'\to S',\varphi)$ is \emph{effective} if it is isomorphic
to the canonical descent datum associated with an \etale{} morphism $X\to S$ as
above. If $S'\to S$ is a morphism of $\catE$-descent, e.g., universally
submersive, then there is at most one morphism $X\to S$ which descends $(X'\to
S',\varphi)$.

We say that $S'\to S$ is a morphism of \emph{effective} $\catE$-descent if
every object $(X'\to S')\in \catE_{/S'}$ equipped with a descent datum is
effective. We say that $S'\to S$ is a morphism of universal $\catE$-descent
(resp.\ universal effective $\catE$-descent) if $S'\times_S T\to T$ is a
morphism of $\catE$-descent (resp.\ effective $\catE$-descent) for any base
change $T\to S$.
\end{xpar}

We briefly state some useful reduction results.

\begin{proposition}[{\cite[Prop.~10.10, Prop.~10.11]{giraud_descente}}]
\label{P:giraud-composition}
Let $\catF$ be a category fibered over $\Sch$. Let $\map{f}{X}{Y}$ and
$\map{g}{Y}{S}$ be morphisms of schemes. If $f$ and $g$ are morphisms of
universal effective $\catF$-descent, then so is $g\circ f$. If $g\circ f$ is a
morphism of universal effective $\catF$-descent, then so is $g$.
\end{proposition}

\begin{proposition}[{\cite[Thm.~10.8 (ii)]{giraud_descente}}]
\label{P:giraud-pullbacks}
Let $\catF$ be a category fibered over $\Sch$. Let $\map{f}{S'}{S}$ be a
morphism of universal $\catF$-descent and let $\map{g}{T}{S}$ be a morphism of
universal effective $\catF$-descent. Let $\map{f'}{T'}{T}$ be the pull-back of
$f$ along $g$. Let $x'\in\catF_{/S'}$ be an object equipped with a descent datum
$\varphi$. Let $y'\in\catF_{/T'}$ and $\varphi_T$ be the pull-back of $x'$ and
$\varphi$ along~$g$. If $(y',\varphi_T)$ is effective, then so is
$(x',\varphi)$. In particular, if $f'$ and $g$ are morphisms of universal
effective $\catF$-descent, then so is $f$.
\end{proposition}

\begin{proposition}\label{P:effectiveness-under-limits}
Let $\catF\subseteq \catE_{\qc}$ be a category fibered over $\Sch$. Let
$S=\Spec(A)$ be affine and let $S'=\varprojlim_\lambda S'_\lambda$ be an
inverse limit
of affine $S$-schemes such that $S'\to S$ is universally submersive. Then
$S'\to S$ is a morphism of universal effective $\catF$-descent if and only if
$S'_\lambda\to S$ is a morphism of universal effective $\catF$-descent for
every $\lambda$.
\end{proposition}
\begin{proof}
The necessity follows from Proposition~\pref{P:giraud-composition}. For
sufficiency, we only need to show effectiveness by
Proposition~\pref{P:univ-sub-is-m.o.d.}. Effectiveness follows easily from the
fact that any object $X'\to S'$ in $\catF_{/S'}$ is of finite presentation. In
fact, there is an index $\lambda$ and a \etale{} morphism $X'_\lambda\to
S'_\lambda$ such that $X'=X'_\lambda\times_{S'_\lambda} S'$,
cf.~\cite[Thm.~8.8.2, Prop.~17.7.8]{egaIV}. If $X'\to S'$ is equipped with
a descent datum, i.e., an $S'\times_S S'$-isomorphism $\map{\varphi}{X'\times_S
S'}{S'\times_S X'}$, then there is $\lambda'\geq \lambda$ such that
$X_{\lambda'}=X_\lambda\times_{S_\lambda} S_{\lambda'}$ has a descent datum
$\varphi_{\lambda'}$ which coincides with the descent datum $\varphi$ after the
pull-back along $S'\times_S S'\to S'_\lambda\times_S S'_\lambda$.
This follows from~\cite[Cor.~8.8.2.5]{egaIV}.
\end{proof}

\begin{proposition}\label{P:effectiveness-under-limits-2}
Let $\catF\subseteq \catE_{\qc}$ be a category fibered over $\Sch$. Let
$S=\Spec(A)$ be affine and $T=\varprojlim_\lambda T_\lambda$ an inverse limit
of affine $S$-schemes. Let $S'\to S$ be a morphism of universal $\catF$-descent
and let $X'\to S'$ be an element of $\catF_{/S'}$ together with a descent datum
$\varphi$. We let $X'_{T_\lambda}\to T'_\lambda=T_\lambda\times_S S'$ be the
pull-back of $X'\to S'$ along $T_\lambda\to S$ and $\varphi_{T_\lambda}$ the
corresponding descent datum. We define $X'_T$ and $\varphi_T$ in the obvious
way.
If $(X'_T,\varphi_T)$ is effective, then $(X'_{T_\lambda},\varphi_{T_\lambda})$
is effective for some index $\lambda$.
\end{proposition}
\begin{proof}
As $(X'_T,\varphi_T)$ is effective there is an \etale{} and quasi-compact
morphism $X_T\to T$ together with an isomorphism $X_T\times_T T'\iso X'_T$
compatible with the descent datum. As $X_T\to T$ is of finite presentation,
there is an index $\lambda$ and an \etale{} and quasi-compact scheme
$X_{T_\lambda}\to T_\lambda$. After increasing $\lambda$, we can assume
that there is an isomorphism $X_{T_\lambda} \times_{T_\lambda} T'_\lambda
\iso X'_{T_\lambda}$ and that this is compatible with the descent datum
$\varphi_{T_\lambda}$.
\end{proof}

The following proposition is an immediate consequence of effective
fpqc-descent for quasi-affine schemes~\cite[Exp.~VIII, Cor.~7.9]{sga1}
as separated, quasi-compact and \etale{} morphisms are quasi-affine by
Zariski's main theorem~\cite[Thm.~8.12.6]{egaIV}.

\begin{proposition}[{\cite[Exp.~IX, Prop.~4.1]{sga1}}]\label{P:fpqc-descent}
Let $\map{f}{S'}{S}$ be faithfully flat and quasi-compact. Then $f$ is
a morphism of universal \emph{effective} $\catE_{\sep,\qc}$-descent.
\end{proposition}

\begin{proposition}\label{P:fppf-descent}
Let $\map{f}{S'}{S}$ be faithfully flat and locally of finite presentation.
Then $f$ is a morphism of universal \emph{effective} $\catE$-descent.
\end{proposition}
\begin{proof}
This follows from~\cite[Cor.~10.4.2]{laumon}.
\end{proof}

\begin{proposition}\label{P:locality-of-effectiveness}
Let $\map{f}{S'}{S}$ be a universally submersive morphism and let $X'\to S'$ be
an object in $\catE_{\qc}$ equipped with a descent datum $\varphi$. For
any morphism $T\to S$ we let $X'_T\to T'=T\times_S S'$ be the pull-back of
$X'\to S'$ and $\varphi_T$ the corresponding descent datum. The following
are equivalent:
\begin{enumerate}
\item $(X',\varphi)$ is effective.
\item $(X'_T,\varphi_T)$ is effective for every $T$ such that
$T=\Spec(\sO_{S,s})$ for some $s\in S$.
\item $(X'_T,\varphi_T)$ is effective for every $T$ such that
$T=\Spec(\shensel{\sO_{S,s}})$ is the strict henselization of $S$ at
some point $s\in S$.
\end{enumerate}
If in addition $S$ is locally noetherian and $X'\to S'$ is in
$\catE_{\qc,\sep}$ then these statements are equivalent to the following:
\begin{enumerate}\setcounter{enumi}{3}
\item $(X'_T,\varphi_T)$ is effective for every $T$ such that
$T=\Spec(\widehat{\sO_{S,s}})$ is the completion of $S$ at some point $s\in S$.
\end{enumerate}
\end{proposition}
\begin{proof}
It is clear that (i)$\implies$(ii)$\implies$(iii)$\implies$(iv). Assume that
(ii) holds, then it follows from
Proposition~\pref{P:effectiveness-under-limits-2} that $(X',\varphi)$ is
effective in an open neighborhood of any point. In particular, there is an open
covering $T\to S$ such that $(X'_T,\varphi_T)$ is effective. By
Proposition~\pref{P:fppf-descent} and Proposition~\pref{P:giraud-pullbacks} it
follows that $(X',\varphi)$ is effective. Similarly, if (iii) holds, there is
an \etale{} covering over which $(X',\varphi)$ is effective and we can again
conclude that $(X',\varphi)$ is effective by Proposition~\pref{P:fppf-descent}.
Finally (iv)$\implies$(iii) by Proposition~\pref{P:fpqc-descent}.
\end{proof}

\begin{remark}\label{R:popescu}
If $S$ is \emph{excellent}, then $\widehat{\sO_{S,s}}$ is a direct limit of
smooth $\sO_{S,s}$-algebras by Popescu's
theorem~\cite{swan_popescus_theorem,spivakovsky_popescus_theorem}. It thus
follows from Propositions~\pref{P:effectiveness-under-limits-2}
and~\pref{P:fppf-descent} that (iv) implies (i) also for the fibered category
$\catE_{\qc}$ when $S$ is excellent. We will not use this fact.
\end{remark}

\begin{proposition}[{\cite[Exp.~IX, Thm.~4.12]{sga1}}]
\label{P:proper-descent-for-etfin}
Proper surjective morphisms of finite presentation are morphisms of
effective $\catE_{\finite}$-descent.
\end{proposition}
\begin{proof}
Let $\map{f}{S'}{S}$ be a proper and surjective morphism of finite
presentation. To show that $f$ is a morphism of effective descent, we can
assume that $S$ is affine by Proposition~\pref{P:locality-of-effectiveness}. As
$S'\to S$ and the morphisms of $\catE_{\finite}$ are of finite presentation, we
can by a limit argument reduce to the case where $S$ is
noetherian. By~Proposition~\pref{P:locality-of-effectiveness} we can further
assume that $S$ is the spectrum of a complete noetherian local ring.

Let $S_0$ be the closed point of $S$ and let $S'_0$, $S''_0$ and $S'''_0$ be
the fibers of $S'$, $S''$ and $S'''$ over $S_0$. By
Theorem~\pref{T:1-henselian-for-proper-over-complete}, the morphisms
$S_0\inj S$, $S'_0\inj S'$, etc., induce equivalences between the category
of finite \etale{} covers over the source and the category of finite
\etale{} covers over the target. Thus $\map{f}{S'}{S}$ is a morphism of
effective descent for $\catE_{\finite}$ if and only if $\map{f_0}{S'_0}{S_0}$
is of effective descent. But $f_0$ is flat and hence of effective
descent by Proposition~\pref{P:fpqc-descent}.
\end{proof}

\begin{corollary}\label{C:proper-descent-for-etsepqc}
Proper and surjective morphisms of finite presentation are morphisms of
effective descent for $\catE_{\sep,\qc}$.
\end{corollary}
\begin{proof}
Let $\map{f}{S'}{S}$ be a proper and surjective morphism of finite
presentation. To show that $f$ is a morphism of effective descent, we can as in
the proof of Proposition~\pref{P:proper-descent-for-etfin} assume that $S$ is
the spectrum of a noetherian local ring. In particular, we can assume that $S$
is noetherian and of finite dimension. We will now prove effectiveness using
induction on the dimension of $S$.

Let $n=\dim(S)$ and assume that every proper surjective morphism of finite
presentation $T'\to T$ such that $\dim(T)<n$ is a morphism of effective
descent. If $n<0$ then it is clear that $f$ is effective. 
By Proposition~\pref{P:locality-of-effectiveness}, it is enough to show
effectiveness for the completion of every local ring of $S$. We can thus
assume that $S$ is a complete local noetherian ring of dimension at most $n$.

Let $X'\to S'$ be a quasi-compact and separated \etale{} morphism. Let $S_0$
be the closed point of $S$. Let $S'_0$ and $X'_0$ be the inverse images of
$S_0$. As $S'_0\to S_0$ is fpqc, there exists $X_0\to S_0$ such that $X'_0\to
S'_0$ is the pull-back. Clearly $X_0\to S_0$ is finite and hence $X'_0\to S'_0$
is finite. Thus $X'_0\to S_0$, $X''_0\to S_0$ and $X'''_0\to S_0$ are proper.

By~\cite[Cor.~5.5.2]{egaIII} there are thus canonical decompositions into open
disjoint subsets $X'=Z'\amalg U'$, $X''=Z''\amalg U''$ and $X'''=Z'''\amalg
U'''$ such that $Z'$, $Z''$ and $Z'''$ are proper over $S$ and contain $X'_0$,
$X''_0$ and $X'''_0$ respectively. Replacing $X'$ with $Z'$ or $U'$ we can thus
assume that either $X'$ is finite over $S'$ or that $X'_0$ is empty. In the
first case it follows that $f$ is effective from
Proposition~\pref{P:proper-descent-for-etfin}. In the second case we can
replace $S$ with $S\setminus S_0$ which has dimension at most $n-1$. It then
follows from the induction hypothesis that $f$ is effective.
\end{proof}

The proof of the following generalization is similar to and independent of
Corollary~\pref{C:proper-descent-for-etsepqc}. As the methods are less
standard and involve algebraic spaces, we have chosen to not state
Corollary~\pref{C:proper-descent-for-etsepqc} as a corollary
of~\pref{C:proper-descent-for-etqc}.

\begin{corollary}\label{C:proper-descent-for-etqc}
Proper and surjective morphisms of finite presentation are morphisms of
effective descent for $\catE_{\qc}$.
\end{corollary}
\begin{proof}
Let $\map{f}{S'}{S}$ be a proper and surjective morphism of finite
presentation. As in the proof of Corollary~\pref{C:proper-descent-for-etsepqc}
we can reduce to $S$ noetherian and of finite dimension and we proceed by
induction on the dimension of $S$.

Let $n=\dim(S)$ and assume that every proper surjective morphism of finite
presentation $T'\to T$ such that $\dim(T)<n$ is a morphism of effective
descent. If $n<0$ then it is clear that $f$ is effective. By
Proposition~\pref{P:locality-of-effectiveness}, it is enough to show
effectiveness for the strict henselization of every local ring of $S$. We can
thus assume that $S$ is a strictly local noetherian ring of dimension at most
$n$.

Let $X'\to S'$ be a quasi-compact \etale{} morphism. Let $S_0$
be the closed point of $S$. Let $S'_0$ and $X'_0$ be the inverse images of
$S_0$. As $S'_0\to S_0$ is fppf, there exists $X_0\to S_0$ such that $X'_0\to
S'_0$ is the pull-back. As $S_0$ is the spectrum of a separably closed field,
we have that $X_0$ is a disjoint union of $m$ copies of $S_0$. Let
$s_1,s_2,\dots,s_m$ be the corresponding sections of $X_0/S_0$ and let
$s'_i$, $s''_i$, $s'''_i$ be the corresponding sections of $X'_0/S'_0$ etc.

As $(S',S'_0)$ (resp.\ $(S'',S''_0)$ etc.) are $0$-henselian pairs by
Proposition~\pref{P:0-henselian-for-proper-over-henselian}, the sections
$s'_i$, (resp.\ $s''_i$ etc.) uniquely lift to sections of $X'/S'$ (resp.\ 
$X''/S''$ etc.) by Proposition~\pref{P:hensel-equiv}. Let $Z'=S'^{\amalg m}$ and
$U'=X'\times_{S'} S'\setminus S'_0$ and similarly for $Z''$, $U''$ etc. From
the sections we obtain canonical open coverings $Z'\amalg U'\to X'$ (resp.\ 
$Z''\amalg U''\to X''$ etc.). By the induction hypothesis it follows that $f$ is
a morphism of effective descent for $U'$ so that $U'\to S'\setminus S'_0$
descends to an \etale{} morphism $U\to S\setminus S_0$ with sections $s_{i,0}$.
We let $X$ be the algebraic space given by gluing a copy of $S$ to $U$ along
$s_{i,0}(S\setminus S_0)$ for each $i$ so that $Z\amalg U\to X$ is an open
covering where $Z=S^{\amalg m}$. Then $X\to S$ descends $X'\to S'$.
\end{proof}


\begin{theorem}\label{T:descent-for-etqc}
Quasi-compact universally subtrusive morphisms are morphisms of
effective $\catE_{\qc}$-descent.
\end{theorem}
\begin{proof}
Let $\map{f}{S'}{S}$ be a universally subtrusive quasi-compact morphism. To
show that $f$ is effective we can, replacing $S$ and $S'$ by open covers,
assume that $S$ and $S'$ are affine.
Proposition~\pref{P:effectiveness-under-limits} shows that it is enough to show
effectiveness for finitely presented $f$. Such an $f$ has a refinement which is
a composition of a flat morphism followed by a proper morphism by
Theorem~\pref{T:ST-aff}. The theorem thus follows from
Propositions~\pref{P:giraud-composition},
\pref{P:fppf-descent} and Corollary~\pref{C:proper-descent-for-etqc}.
\end{proof}

As a corollary we answer a question posed by
Grothendieck~\cite[Exp.~IX, Comment after Cor.~3.3]{sga1} affirmatively.

\begin{corollary}
Universal submersions between noetherian schemes are morphisms of effective
descent for $\catE_{\finite}$, $\catE_{\sep,\qc}$ and $\catE_{\qc}$.
\end{corollary}

As another corollary we have the following result:

\begin{theorem}\label{T:descent-for-et}
The following classes of morphisms are classes of effective $\catE$-descent.
\begin{enumerate}
\item Universally open and surjective morphisms.
\item Universally closed and surjective morphisms of finite presentation.
\item Universally subtrusive morphisms of finite presentation.
\end{enumerate}
\end{theorem}
\begin{proof}
First note that the morphisms in the first two classes are universally
subtrusive, cf.\ Remark~\pref{R:subtrusive-examples}. Moreover by
Theorem~\pref{T:ST-voevodsky}, a morphism in the third class has a refinement
which is the composition of a morphism in the first class and a morphism in the
second class. Thus, it is enough to prove the effectiveness of the first two
classes by Proposition~\pref{P:giraud-composition}.

Let $\map{f}{S'}{S}$ be
either a universally open morphism or a universally closed morphism of finite
presentation. Let $X'\to S'$ be an \etale{} morphism equipped with a descent
datum. By Proposition~\pref{P:locality-of-effectiveness}, we can assume that
$S$ is affine. If $f$ is universally open, then there is an open quasi-compact
subset $U$ of $S'$ such that $f|_U$ is surjective. Replacing $S'$ with $U$ we
can assume that $f$ is quasi-compact.

Let $x'\in X'$. If $f$ is universally open, let $V\subseteq X'$ be an open
quasi-compact neighborhood of $x'$ and let $R(V)=\pi_2(\pi_1^{-1}(V))$ be the
saturation of $V$ with respect to the equivalence relation
$\map{R=(\pi_1,\pi_2)}{X''}{X'\times_S X'}$. As the $\pi_i$'s are quasi-compact
and open, we have that $U=R(V)$ is an open quasi-compact $R$-stable
neighborhood of $x'$.

If $f$ is universally closed and finitely presented, let
$R(x')=\pi_2(\pi_1^{-1}(x'))\subseteq X'$ be the saturation of $x'$. As $\pi_i$
is quasi-compact, the subset $R(x')$ is quasi-compact. Let $V$ be an open
quasi-compact neighborhood of $R(x')$. As $X'$ is quasi-separated,
we have that $V\subseteq X'$ is retrocompact and thus pro-constructible. In
particular the complement $X'\setminus V$ is ind-constructible. As $\pi_i$ is
closed and finitely presented, the saturation $R(X'\setminus V)$ is a closed
ind-constructible subset of $X'$ disjoint from $R(x')$. Thus
$U=X'\setminus R(X'\setminus V)$ is an open $R$-stable pro-constructible
neighborhood of $x'$ contained in $V$. In particular $U\subseteq V$ is
retrocompact and hence $U$ is quasi-compact.

In both cases, we thus have an open covering $\coprod U_{x'}\to X'$, stable
under the descent datum, such that each $U_{x'}$ is quasi-compact.
By Theorem~\pref{T:descent-for-etqc} every space $U_{x'}$
descends to an \etale{} quasi-compact space $U_x$ over $S$.
As $U_{x'}\to U_x$ is submersive, the intersection $U_{x_1'}\cap U_{x_2'}$
descends to an open subset of both $U_{x_1}$ and $U_{x_2}$. Finally as $S'\to
S$ is a morphism of $\catE$-descent, the gluing datum of the $U_{x'}$'s
descends to a gluing datum of the $U_{x}$'s. Thus the $U_{x}$'s glue to an
algebraic space $X$ \etale{} over $S$ which descends $X'\to S'$.
\end{proof}


Recall that a morphism of algebraic spaces $\map{f}{X}{Y}$ is a \emph{universal
homeomorphism} if $\map{f'}{X\times_Y Y'}{Y'}$ is a homeomorphism (of
topological spaces) for every morphism $Y'\to Y$ of algebraic spaces. As usual,
it is enough to consider base changes such that $Y'$ is a scheme or even
an affine scheme.

The diagonal of a universally injective morphism is surjective.
Thus,
any homeomorphism of \emph{schemes} is separated. It then follows
from Zariski's main theorem that a finite type morphism of schemes is a
universal homeomorphism if and only if it is universally injective, surjective
and finite. More generally, a morphism of schemes $\map{f}{X}{Y}$ is a
universal homeomorphism if and only if $f$ is universally injective, surjective
and integral~\cite[Cor.~18.12.11]{egaIV}.

The diagonal of a universal homeomorphism of \emph{algebraic spaces} is a
surjective monomorphism but not necessarily an immersion. In particular, not
every homeomorphism of algebraic spaces is separated. This is demonstrated by
the following classical example of an algebraic space which is not locally
separated, i.e., its diagonal is not an immersion.

\begin{example}[{\cite[Ex.~1, p.~9]{knutson_alg_spaces}}]
Let $U$ be the union of two secant affine lines and let $R$ be the equivalence
relation on $U$ which identifies the two lines except at the singular
point. Then the quotient $X=U/R$ is an algebraic space whose underlying
topological space is the affine line. In fact, there is a universal
homeomorphism $X\to \A{1}$ such that $U\to X\to \A{1}$ induces the identity on
the two components. The corresponding two sections $\A{1}\to X$ are bijective
but not universally closed.
The space $X$ looks like the affine line except at a special point where it has
two different tangent directions.
\end{example}

The following theorem generalizes~\cite[Exp.~IX, Thm.~4.10]{sga1}:

\begin{theorem}\label{T:etale-and-homeomorphisms}
Let $S'\to S$ be a \emph{separated} universal homeomorphism of algebraic
spaces. Then the functor $\catE_{/S}\to \catE_{/S'}$
$$\xymatrix@R=0pt{
{\{ \text{\etale{} spaces over $S$} \}}\ar[r] &
  {\{ \text{\etale{} spaces over $S'$} \}} \\
{X}\ar@{|->}[r] & {X\times_S S'}}$$
is an equivalence of categories. In particular, we have induced equivalences of
categories $\catF_{/S}\to \catF_{/S'}$ where $\catF$ is one of the fibered
categories $\catE_{\finite}$, $\catE_{\sep,\qc}$, $\catE_{\qc}$,
$\catE_{\sep}$.
%
\end{theorem}
\begin{proof}
As $S'\to S$ is a separated universal homeomorphism $S'\inj S'\times_S S'$ is a
nil-immersion. The functor from \etale{} algebraic spaces over $S'\times_S S'$
to \etale{} algebraic spaces over $S'$ is therefore an equivalence by
Proposition~\pref{P:nil-imm-equiv-etale}. In particular, every \etale{}
algebraic space over $S'$ comes with a unique descent datum. This shows that
the functor in the theorem is fully faithful.

Essential surjectivity for $\catF=\catE_{\qc}$ follows from
Theorem~\pref{T:descent-for-etqc}. For an object $X'$ in the category
$\catE_{/S'}$ we first choose an open covering $\{U'_\alpha\}$ of $X'$ such
that the $U'_\alpha$'s are quasi-compact spaces. These $U'_\alpha$'s come with
unique descent data and can be descended to $S$. As in the last part of the
proof of Theorem~\pref{T:descent-for-et} we can glue the descended spaces to
an algebraic space which descends $X'$.
%
%
\end{proof}

\begin{corollary}\label{C:sep-homeo-is-affine}
A separated universal homeomorphism of algebraic spac\-es is representable by
schemes and is integral.
\end{corollary}
\begin{proof}
Let $S'\to S$ be a separated universal homeomorphism and choose an \etale{}
presentation $U'\to S'$ such that $U'$ is a \emph{scheme}. Then by
Theorem~\pref{T:etale-and-homeomorphisms} this induces an \etale{} morphism
$U\to S$ such that $U'=U\times_S S'$. As $U'\to U$ is a representable universal
homeomorphism it is integral by~\cite[Cor.~18.12.11]{egaIV}. In particular
$U'\to U$ is affine and it follows by \etale{} descent that $S'\to S$ is
representable and integral.
\end{proof}

\begin{remark}
The proof of Corollary~\pref{C:sep-homeo-is-affine} shows that
Theorem~\pref{T:etale-and-homeomorphisms} is false for non-separated universal
homeomorphisms.
\end{remark}


\begin{example}[Push-outs]
Let $Z\inj X$ be a closed immersion of affine schemes and let $Z\to Y$ be any
morphism of affine schemes. Then the push-out $X\amalg_Z Y$ exists in the
category of schemes and is affine~\cite[Thm.~5.1]{ferrand_conducteur}.
Furthermore $Z=X\times_{X\amalg_Z Y} Y$ and $\map{f}{X\amalg Y}{X\amalg_Z Y}$
is universally submersive~\cite[Thm.~7.1~A)]{ferrand_conducteur}.
Let $E\to X\amalg Y$ be an affine \etale{} morphism equipped with a
descent datum with respect to $f$. The descent datum gives in particular
an isomorphism $E|_X \times_X Z \to E|_Y \times_Y Z$. We can then form the
push-out $E|_X \amalg_{E|_Z} E|_Y$ which is affine and \etale{}
over $X\amalg_Z Y$ and descends $E$~\cite[Thm.~2.2 (iv)]{ferrand_conducteur}.
Thus, $f$ is a morphism of effective descent for the category of affine and
\etale{} morphisms.

In general, $\map{f}{X\amalg Y}{X\amalg_Z Y}$ is not
\emph{subtrusive}. For example, let $Z\to Y$ be the open immersion
$\A{1}_x\setminus \{x=0\}\subseteq \A{1}_x$, let $X=\A{2}_{x,y}\setminus
\{x=0\}$ and let $Z\inj X$ be the hyperplane defined by $y=0$. Then
$Y\inj X\amalg_Z Y$ is a closed immersion, $X\to X\amalg_Z Y$ is an
open immersion and $X$ and $Y$ intersect along the locally closed subset
$Z$. The ordered pair $\{x=y=0\} < \xi$ of points on $X\amalg_Z Y$, where
$\xi$ is the generic point on $Y$, cannot be lifted to $X\amalg Y$.
\end{example}

This example motivates the following question:

\begin{question}
Are quasi-compact universally \emph{submersive} morphisms of effective
$\catE_{\qc}$-descent?
\end{question}

\end{section}


\begin{section}{Passage to the limit}\label{S:limit}
In this section, we first show that subtrusive morphisms are \emph{stable}
under inverse limits. This result follows from basic properties of subtrusive
morphisms (\ref{P:S-subtrusive}--\ref{P:locality-of-subtrusiveness}).
The corresponding stability result for universally open morphisms is proved
in~\cite[Prop.~8.10.1]{egaIV}. We then show that subtrusive morphisms and
universally open morphisms \emph{descend} under inverse limits. The proofs of
these results are much more difficult and use the structure theorems of
Section~\ref{S:structure-theorems}.

\begin{notation}\label{N:limit}
We use the following notation, cf.~\cite[\S8]{egaIV}: Let $S_0$ be a scheme and
let $S_\lambda$ be a filtered inverse system of schemes, affine over $S_0$. Let
$S=\varprojlim_\lambda S_\lambda$ be the inverse limit which is a scheme
affine over $S_0$. Let $\alpha$ be an index and let
$\map{f_\alpha}{X_\alpha}{Y_\alpha}$ be a morphism of $S_\alpha$-schemes. For
every $\lambda\geq\alpha$ we let $\map{f_\lambda}{X_\lambda}{Y_\lambda}$ be the
pull-back of $f_\alpha$ along $S_\lambda\to S_\alpha$ and we let
$\map{f}{X}{Y}$ be the pull-back of $f_\alpha$ along $S\to S_\alpha$. Let
$\map{u_\lambda}{X}{X_\lambda}$ and $\map{v_\lambda}{Y}{Y_\lambda}$ be the
canonical morphisms.
\end{notation}

\begin{proposition}[{\cite[Part~II, Prop.~3]{picavet_submersion}}]
\label{P:subtrusiveness-stable-under-limits}
Let $f$ and $f_\lambda$ be morphisms as in Notation~\pref{N:limit} and assume
that $f^{\cons}$ is submersive, e.g. $f$ quasi-compact.
If there exists $\lambda$ such that $f_\mu$ is
subtrusive (resp.\ universally subtrusive) for every $\mu\geq\lambda$, then $f$
is subtrusive (resp.\ universally subtrusive).
\end{proposition}
\begin{proof}
If $f_\lambda$ is universally subtrusive then it follows from the definition
that the pull-back $f$ is universally subtrusive. Assume that there is
$\lambda$ such that $f_\mu$ is subtrusive for $\mu\geq\lambda$. To prove that
$f$ is subtrusive, it is enough to show that if $Z\subseteq Y$ is
pro-constructible, then $\overline{Z}=f(\overline{f^{-1}(Z)})$ by
Proposition~\pref{P:S-subtrusive}. Let $Z_\mu=v_\mu(Z)$ which is
pro-constructible as $v_\mu$ is quasi-compact. Then
$$Z=\bigcap_{\mu\geq\lambda} v^{-1}_\mu(Z_\mu)$$
and as $Y=\varprojlim_\mu Y_\mu$ as topological spaces, it follows that
$$\overline{Z}=\bigcap_{\mu\geq\lambda} v^{-1}_\mu
\left(\overline{Z_\mu}\right).$$
Similarly
$$\overline{f^{-1}(Z)}=\bigcap_{\mu\geq\lambda} u^{-1}_\mu
\left(\overline{f_\mu^{-1}(Z_\mu})\right).$$
As $f_\mu$ is subtrusive we have that
$\overline{Z_\mu}=f_\mu(\overline{f^{-1}_\mu(Z_\mu)})$. It thus
follows that
$$\overline{Z}=\bigcap_{\mu\geq\lambda} v^{-1}_\mu
\left(f_\mu\left(\overline{f^{-1}_\mu(Z_\mu)}\right)\right)
=\bigcap_{\mu\geq\lambda} f\left(u^{-1}_\mu
\left(\overline{f_\mu^{-1}(Z_\mu})\right)\right)
=f\left(\overline{f^{-1}(Z)}\right)$$
as the intersections are filtered.
\end{proof}

\begin{corollary}\label{C:univ-submersive-An}
Let $\map{f}{X}{Y}$ be a morphism of schemes. Then $f$ is universally
subtrusive if and only if $f^{\cons}$ is universally submersive and
$\map{f_n}{X\times_{\Z} \A{n}_{\Z}}{Y\times_{\Z}
\A{n}_{\Z}}$ is subtrusive for every positive integer $n$.
\end{corollary}
\begin{proof}
The condition is necessary by the definition of universally subtrusive. For the
sufficiency, assume that $f_n$ is subtrusive for all $n$. As subtrusiveness is
Zariski-local on the base by Proposition~\pref{P:locality-of-subtrusiveness},
we can assume that $Y$ is affine. It is also enough to check that
$\map{f'}{X'}{Y'}$ is subtrusive for base changes $Y'\to Y$ such that $Y'$ is
affine. First assume that $Y'\to Y$ is of finite type. Then we can factor
$Y'\to Y$ through a closed immersion $Y'\inj Y\times \A{n}$ and it follows by
the assumptions on $f_n$ and Proposition~\pref{P:locality-of-subtrusiveness}
that $f'$ is subtrusive. For arbitrary affine $Y'\to Y$, we write $Y'$ as a
limit of finite type schemes and invoke
Proposition~\pref{P:subtrusiveness-stable-under-limits}.
\end{proof}

\begin{theorem}\label{T:subtrusive-limit-of-subtrusive}
Assume that $S_0$ is quasi-compact and $\map{f_\alpha}{X_\alpha}{Y_\alpha}$ is
of finite presentation with notation as in~\pref{N:limit}. Then $\map{f}{X}{Y}$
is universally subtrusive if and only if $f_\lambda$ is universally subtrusive
for some $\lambda\geq \alpha$.
\end{theorem}
\begin{proof}
The condition is sufficient by definition. To prove the necessity we assume
that $f$ is universally subtrusive. As $S_0$ is quasi-compact there is a finite
affine covering of $S_0$. As subtrusiveness is local on the base by
Proposition~\pref{P:locality-of-subtrusiveness} we can therefore assume that
$S_0$ is affine. We can then by Theorem~\pref{T:ST-aff} find a refinement
$\map{f'}{X'}{Y}$ of $\map{f}{X}{Y}$ such that $f'$ has a factorization into a
finitely presented flat surjective morphism $X'\to Y'$ followed by a finitely
presented proper surjective morphism $Y'\to Y$. By~\cite[Thm.~8.10.5 and
Thm.~11.2.6]{egaIV} the morphism $f'$ descends to a morphism
$\map{f'_\lambda}{X'_\lambda}{Y_\lambda}$ with a similar factorization. In
particular $f'_\lambda$ is universally subtrusive and it follows that
$f_\lambda$ is universally subtrusive as well.
\end{proof}

\begin{corollary}\label{C:subtrusive-reduction-to-exc-noeth}
Let $S=\Spec(A)$ be an affine scheme and let $\map{f}{X}{S}$ be a morphism
of finite presentation. Then the following are equivalent:
\begin{enumerate}
\item $f$ is universally subtrusive.
\item There exists an affine noetherian scheme $S_0=\Spec(A_0)$, a morphism
$\map{f_0}{X_0}{S_0}$ of finite presentation and a morphism $S\to S_0$ such
that $X=X_0\times_{S_0} S$ and $f_0$ is universally
submersive.\label{CI:sren-ii}
\item There exists a scheme $S_0$ and a morphism $f_0$ as in~\ref{CI:sren-ii}
such that in addition $A_0$ is a sub-$\Z$-algebra of $A$ of finite type.
\end{enumerate}
\end{corollary}

The corresponding result for universally submersive is false. Indeed, there
exists
a universally submersive morphism of finite presentation which is not
universally subtrusive, cf.\ Example~\pref{E:submersive-non-subtrusive}, and
hence not a pull-back from a universally submersive morphism of noetherian
schemes.

\begin{theorem}
Assume that $S_0$ is quasi-compact and $\map{f_\alpha}{X_\alpha}{Y_\alpha}$ is
of finite presentation with notation as in~\pref{N:limit}. Then $\map{f}{X}{Y}$
is universally open if and only if $f_\lambda$ is universally open for some
$\lambda\geq \alpha$.
\end{theorem}
\begin{proof}
As the condition is clearly sufficient, we assume that $f$ is universally open.
As $S_0$ is quasi-compact we can easily reduce to the case where $Y_\alpha$ is
affine. Using~\cite[Thm.~8.10.5 and Thm.~11.2.6]{egaIV} we can then descend the
refinement of $f$ given by Theorem~\pref{T:ST-aff}. There is thus an index
$\lambda$, a proper surjective morphism $Y'_\lambda\to Y_\lambda$, a faithfully
flat morphism of finite presentation $X'_\lambda\to Y'_\lambda$ and a
nil-immersion $X'_\lambda\inj X_\lambda\times_{Y_\lambda} Y'_\lambda$. As
$X'_\lambda\to Y'_\lambda$ is universally open, so is
$X_\lambda\times_{Y_\lambda} Y'_\lambda\to Y'_\lambda$. As $Y'_\lambda\to
Y_\lambda$ is universally submersive it follows that $X_\lambda\to Y_\lambda$
is universally open.
\end{proof}

\end{section}


\begin{section}{Weakly normal descent}\label{S:descent-of-morphisms}
Let $\map{f}{S'}{S}$ be faithfully flat and quasi-compact. Then $f$ is a
morphism of descent for the fibered category of all morphisms of algebraic
spaces, that is, for any algebraic space $X$ we have that
$$\xymatrix{
\Hom(S,X)\ar[r] & \Hom(S',X)\ar@<.5ex>[r]\ar@<-.5ex>[r]
 & \Hom(S'\times_S S',X)}$$
is exact~\cite[Thm.~A.4]{laumon}. In this section, we give a similar descent
result for \emph{weakly normal} universally submersive morphisms.

\begin{xpar}[Schematic image]
Let $\map{f}{X}{Y}$ be a morphism of algebraic spaces. If there exists a
smallest closed subspace $Y'\inj Y$ such that $f$ factors through $Y'\inj
Y$, then we say that $Y'$ is the \emph{schematic image} of
$f$~\cite[6.10]{egaI_NE}. If $X$ is reduced, then $\overline{f(X)}$ with its
reduced structure is the schematic image.

Let $\map{f}{X}{Y}$ be a quasi-compact and quasi-separated morphism of
\emph{algebraic spaces}. Then $f_*\sO_X$ is a quasi-coherent sheaf
(\cite[Prop.~II.4.6]{knutson_alg_spaces} holds for non-separated
morphisms) and the schematic image of $f$ is the closed subspace of $Y$
defined by the ideal $\ker(\sO_Y\to f_*\sO_X)$.
The underlying topological space of the image is $\overline{f(X)}$,
as can be checked on an \etale{} presentation.
\end{xpar}

A morphism $\map{f}{X}{Y}$ of algebraic spaces is \emph{schematically dominant}
if $\sO_Y\to f_*\sO_X$ is injective (in the small \etale{} site). This agrees
with the usual definition for schemes~\cite[D\'ef.~11.10.2, Thm.~11.10.5\ 
(ii)]{egaIV}. If $\map{f}{X}{Y}$ is schematically dominant, then the schematic
image of $f$ exists and equals $Y$. Conversely, if $f$ is quasi-compact and
quasi-separated or $X$ is reduced, then $f$ is schematically dominant if and
only if the schematic image of $f$ equals $Y$.

\begin{proposition}\label{P:sch-dom+univ-sub:epimorphism}
Let $\map{p}{S'}{S}$ be a schematically dominant universally submersive
morphism of algebraic spaces. Then $p$ is an epimorphism in the category of
algebraic spaces, i.e., $\Hom(S,X)\to \Hom(S',X)$ is injective for every
algebraic space $X$.
\end{proposition}
\begin{proof}
First assume that $X$ is \emph{separated} and let $\map{f}{S}{X}$ be a
morphism. Then the schematic image of $\map{\Gamma_f\circ p=(p,f\circ
p)}{S'}{S\times X}$ exists and equals the graph
$\Gamma_f$~\cite[Prop.~6.10.3]{egaI_NE}. We can thus recover $f$ from $f\circ
p$.

For general $X$, let $\map{f_1,f_2}{S}{X}$ be two morphisms such that $f_1\circ
p=f_2\circ p$. Let $U\to X$ be an \etale{} surjective morphism such that $U$ is
a separated scheme. As $p$ is universally submersive, it is a morphism of
descent for \etale{} morphisms by Proposition~\ref{P:univ-sub-is-m.o.d.}. Thus,
the canonical $S'$-isomorphism $V':=p^*f_1^*U\iso p^*f_2^*U$ descends to an
$S$-isomorphism $V:=f_1^* U\iso f_2^* U$. To conclude, we have a diagram
$$\xymatrix{%
  V'\ar[r]^{q}\ar[d] & V\ar@<.5ex>[r]^-{g_1}\ar@<-.5ex>[r]_-{g_2}\ar[d]
  & U\ar[d]\\
  S'\ar[r]^{p} & S\ar@<.5ex>[r]^-{f_1}\ar@<-.5ex>[r]_-{f_2} & X%
}$$
where the vertical morphisms are \etale{}, the natural squares are cartesian
and $g_1\circ q=g_2\circ q$. Note that $q$ is schematically dominant as $p$ is
schematically dominant and $V\to S$ is \etale{}. We apply the special case of
the proposition to deduce that $g_1=g_2$ and it follows that $f_1=f_2$.
\end{proof}

\begin{xpar}[Weak normalization]
Let $\map{f}{S'}{S}$ be a dominant, quasi-compact and quasi-separated
morphism. A \emph{wn-factorization} of $f$ is a factorization $f=f_2\circ f_1$
such that $f_1$ is schematically dominant and $f_2$ is a separated universal
homeomorphism. A wn-factorization is \emph{trivial} if $f_2$ is an isomorphism.
We say that $f$ is \emph{weakly normal} (or weakly subintegrally closed) if any
wn-factorization of $f$ is trivial. The \emph{weak normalization} (or weak
subintegral closure) of $S$ in $S'$, denoted $\wn{S'}{S}$, is the maximal
separated universal homeomorphism ${\wn{S'}{S}\to S}$ such that there exists a
wn-factorization $S'\to \wn{S'}{S}\to S$ of~$f$. There exists a unique weak
normalization and it fits into a unique wn-factorization.
For more details on weakly normal morphisms and the weak normalization, see\
Appendix~\ref{S:weak-norm}.
\end{xpar}

\begin{theorem}\label{T:h-descent}
Let $\map{\pi}{X}{S}$ be a morphism of algebraic spaces and let
$\map{p}{T'}{T}$ be a quasi-compact, quasi-separated, universally submersive
and weakly normal
morphism of algebraic spaces over $S$. Assume either that $X\to S$ is locally
separated (this is the case if $X$ is a scheme) or that $p$ is universally
subtrusive. Then:
$$\xymatrix{%
\Hom_S(T,X)\ar[r] & \Hom_{S}(T',X)\ar@<.5ex>[r]\ar@<-.5ex>[r]
 & \Hom_S((T'\times_T T')_\red,X)}$$
is exact.
\end{theorem}
\begin{proof}
As $p$ is weakly normal,
$p$ is schematically dominant and it follows from
Proposition~\pref{P:sch-dom+univ-sub:epimorphism} that $\Hom_S(T,X)\to
\Hom_{S}(T',X)$ is injective.

Let $T''=(T'\times_T T')_\red$ and let $\map{\pi_1,\pi_2}{T''}{T'}$ denote the
two projections. To show exactness in the middle, let $\map{f'}{T'}{X}$ be
a morphism such that $f'':=f'\circ \pi_1=f'\circ \pi_2$.
%
%
Let $\map{s'=(f',\id{T'})}{T'}{X\times_S T'}$ and
$\map{s''=(f'',\id{T''})}{T''}{X\times_S T''}$ be the induced sections.
Denote the set-theoretical images by $\Gamma'=s'(T')$ and
$\Gamma''=s''(T'')$. As $f''=f'\circ \pi_i$, we have that $s''$ is the
pull-back of $s'$ along either of the two projections $\id{X}\times\pi_i$,
$i=1,2$. In particular, we have that
$\Gamma''=(\id{X}\times\pi_i)^{-1}(\Gamma')$.
Let $\map{p_X=\id{X}\times p}{X\times_S T'}{X\times_S T}$ denote the pull-back
of $p$. Let $\Gamma:=p_X(\Gamma')$ so that $p_X^{-1}(\Gamma)=\Gamma'$.

First assume that $s'$ is a closed immersion so that $\Gamma'$ and $\Gamma''$
are closed. Then $\Gamma$ is also closed since $p_X$ is submersive.
We let ${T_1}$ be the schematic image of the map $\map{(f',p)=p_X\circ
s'}{T'}{X\times_S T}$ so that the underlying set of ${T_1}$ is
$\Gamma$.  Let $\map{q}{T'}{{T_1}}$ be the induced morphism. Then $q$
is surjective and the graph of $q$ is a nil-immersion $T'\to
{T_1}\times_T T'$ since both source and target are closed subspaces of
$X\times_S T'$ with underlying set $\Gamma'$. In particular, it follows that
${T_1}\times_T T'\to T'$ is a separated universal homeomorphism.
We now apply Proposition~\pref{P:descent-of-top-props} to ${T_1}\to T$
and $\map{p}{T'}{T}$ and deduce that ${T_1}\to T$ is universally
closed, separated, universally injective and surjective, i.e., a separated
universal homeomorphism. Since $\map{p}{T'\to{T_1}}{T}$ is weakly
normal, we have that ${T_1}\to T$ is an isomorphism and the morphism
$\map{f}{T={T_1}\inj X\times_S T}{X}$ lifts $f'$.

Instead assume that $X\to S$ is locally separated, i.e., that the diagonal
morphism $\map{\Delta_{X/S}}{X}{X\times_S X}$ is an immersion. Then the
sections $s'$ and $s''$ are also immersions.
The image of an immersion of \emph{algebraic spaces} is locally closed.
Indeed, this follows from taking an \etale{} presentation and
Theorem~\pref{T:loc-closed-2}. Thus $\Delta_{X/S}(X)$, $\Gamma'$ and $\Gamma''$
are locally closed subsets. We will now show that $\Gamma$ is locally
closed. If $p_X$ is universally \emph{subtrusive} this follows from
Theorem~\pref{T:loc-closed-1}.

Let $V\subseteq X\times_S X$ be an open neighborhood of $\Delta(X)$ such that
$\Delta(X)\subseteq V$ is closed. Consider the morphism $\map{(f'\times
\id{X})}{T'\times_S X}{X\times_S X}$. The composition with
either of the two morphism $\pi_i\times \id{X}$ is $f''\times \id{X}$.  Let
$U'=(f'\times\id{X})^{-1}(V)$ and $U''=(f''\times\id{X})^{-1}(V)$ so that if we
let $U=p_X(U')$ then $U'=p_X^{-1}(U)$. The subset
$U\subseteq X\times_S T$ is open since $p_X$ is submersive.
Note that the pull-back of $\Delta_{X/S}$ along $(f'\times \id{X})$ is $s'$.
Therefore $\Gamma'\subseteq U'$ is closed and $f'$ factors
through $U$. After replacing $X$ and $S$ with $U$ and $T$, the section $s'$
becomes a closed immersion so that the previous case applies.

Now, let $X$ be arbitrary and assume that $p$ is universally subtrusive. As the
question is local on $T$, we may assume that $T$ is quasi-compact. After
replacing $X$ with a quasi-compact open $U\subseteq X$ through which $f'$
factors, we can also assume that $X$ is quasi-compact. Let $U\to X$ be an
\etale{} presentation such that $U$ is a quasi-compact scheme. Let
$V'=f'^{-1}(U)$ and $V''_r=f''^{-1}(U)$. By
Proposition~\pref{P:nil-imm-equiv-etale} there is a unique \etale{} $T'\times_T
T'$-scheme $V''$ which restricts to $V''_r$ on $T''$. By
Theorem~\pref{T:descent-for-etqc} the \etale{} morphism $V'\to T'$ descends to
an \etale{} morphism $V\to T$. As the weak normalization commutes with \etale{}
base change, cf.\ Proposition~\pref{P:wn-basechange}, we have that $V'\to V$ is
weakly normal.

We now apply the first case of the theorem to $V'\to V$ and $V'\to U$ and
obtain a morphism $V\to U\to X$ lifting $V'\to X$. Similarly, we obtain a
lifting $V\times_T V\to U\times_X U\to X$ of $V'\times_{T'} V'\to U\times_{X}
U\to X$. Finally, we obtain the morphism $\map{f}{T}{X}$ by \etale{} descent.
\end{proof}

\begin{remark}
Suppose that we remove the assumption that $T'\to T$ is weakly normal in the
theorem. If $X\to S$ is locally separated, then the proof of the theorem shows
that there exists a minimal wn-factorization $T'\to T_1\to T$ such that
$\map{f'}{T'}{X}$ lifts to $T_1$.
%
%
If $X/S$ is locally of finite type, then $T_1\to T$ is of finite type.
%
%
It can be shown that such a minimal wn-factorization also exists
if $X\to S$ is arbitrary and $T'\to T$ is universally subtrusive.
\end{remark}

We obtain the following generalization of Lemma~\pref{L:wn-char}:

\begin{corollary}
Let $\map{p}{S'}{S}$ be a quasi-compact and quasi-separated universally
submersive morphism. Let $\map{q}{(S'\times_S S')_\red}{S}$ be the structure
morphism of the reduced fiber product. Then the sequence
$$\xymatrix{
\sO_{\wn{S'}{S}}\ar@{(->}[r] & p_*\sO_{S'}\ar@<.5ex>[r] \ar@<-.5ex>[r] &
q_*\sO_{(S'\times_S S')_\red}.}$$
is exact. In particular, we have that $p$ is weakly normal if and only if
$$\xymatrix{
\sO_S\ar[r] & p_*\sO_{S'}\ar@<.5ex>[r] \ar@<-.5ex>[r] &
q_*\sO_{(S'\times_S S')_\red}}$$
is exact.
\end{corollary}
\begin{proof}
This follows from the fact that $(S'\times_S
S')_\red=(S'\times_{\wn{S'}{S}} S')_\red$ together with
Theorem~\pref{T:h-descent} applied to $X=\A{1}$.
\end{proof}

\end{section}


\begin{section}{The $h$-topology}\label{S:h-top}
In this section, we look at the $h$- and $qfh$-topologies. An easy description
of the coverings in these topologies is obtained from the structure theorems of
Section~\ref{S:structure-theorems}. In contrast to the Grothendieck topologies
usually applied, the $h$- and $qfh$-topologies are not sub-canonical, i.e., not
every representable functor is a sheaf. It is therefore important to give a
description of the associated sheaf to a representable
functor~\cite{voevodsky_homology,brenner_Groth-top}.

Let $X$ be an algebraic space of finite presentation over a base scheme $S$.
The main result of this section is that the associated sheaf to the functor
$\Hom_S(-,X)$ coincides with the functor $T\mapsto \Hom_S(\wn{}{T},X)$ where
$\wn{}{T}$ is the absolute weak normalization of $T$. This has been proved by
Voevodsky~\cite{voevodsky_homology} when $S$ and $X$ are excellent noetherian
\emph{schemes}.
When $S$ is non-noetherian, it is natural to replace submersive morphisms
with subtrusive morphisms.
To treat the case when $X$ is a general algebraic space, we use the effective
descent results of Section~\ref{S:descent} via Theorem~\pref{T:h-descent}.

Let $S$ be any scheme and let $\Sch_{/S}$ be the category of schemes
over $S$. The following definitions of the $h$- and $qfh$-topologies
generalize~\cite[Def.~3.1.2]{voevodsky_homology} which is restricted to the
category of noetherian schemes.


\begin{definition}\label{D:h-topology}
The $h$-topology is the minimal Grothendieck topology on $\Sch_{/S}$ such that
the following families are coverings
\begin{enumerate}
\item Open coverings, i.e., families of open immersions $\{\map{p_i}{U_i}{T}\}$
such that $T=\bigcup p_i(U_i)$.
\item Finite families $\{\map{p_i}{U_i}{T}\}$ such that
$\map{\coprod p_i}{\coprod U_i}{T}$ is \emph{universally subtrusive} and of
finite presentation.
\end{enumerate}
The $qfh$-topology is the topology generated by the same types of coverings
except that all morphisms should be locally quasi-finite.
\end{definition}

\begin{remark}
The restriction of the $h$-topology (resp.\ $qfh$-topology) to the category of
quasi-compact and quasi-separated schemes is the Grothendieck topology
associated to the pre-topology whose coverings are of the form (ii).
\end{remark}

\begin{remark}
Consider the following types of morphisms:
\begin{enumerate}
\item Finite surjective morphisms of finite presentation.
\item Faithfully flat morphisms, locally of finite presentation.
\item Proper surjective morphisms of finite presentation.
\end{enumerate}
(i) and (ii) are coverings in the $qfh$-topology and (i)--(iii) are coverings
in the $h$-topology. Indeed, morphisms of type (ii) have quasi-finite
flat quasi-sections~\cite[Cor.~17.16.2]{egaIV}.
\end{remark}

The following theorem generalizes~\cite[Thm.~3.1.9]{voevodsky_homology}.

\begin{theorem}\label{T:h-top-desc}
Every $h$-covering (resp.\ $qfh$-covering) $\{U_i\to T\}$ has a refinement of
the form $\{W_{jk}\to W_j \to V_j\to T\}$ such that
\begin{itemize}
\item $\{V_{j}\to T\}$ is an open covering,
\item $W_j\to V_j$ is a proper (resp.\ finite) surjective morphism of finite
presentation for every $j$,
\item $\{W_{jk}\to W_j\}$ is an open quasi-compact covering for every $j$.
\end{itemize}
In particular, the $h$-topology (resp.\ $qfh$-topology) is the minimal
Grothen\-dieck topology such that the following families are coverings:
\begin{enumerate}
\item Families of open immersions $\{\map{p_i}{U_i}{T}\}$
such that $T=\bigcup p_i(U_i)$.
\item Families $\{\map{p}{U}{T}\}$ consisting of a single proper
(resp.\ finite) surjective morphism of finite presentation.
\end{enumerate}
\end{theorem}
\begin{proof}
By~\cite[Exp.~IV, Prop.~6.2.1]{sga3}
it follows that there is a refinement of the form $\{W'_j\to V_j\to T\}$ where
$W'_j\to V_j$ are $h$-coverings (resp.\ $qfh$-coverings) of affine schemes and
$\{V_j\to T\}$ is an open covering. Theorems~\pref{T:STf-noeth/aff}
and~\pref{T:ST-voevodsky} then show that these coverings have a further
refinement as in the theorem.
\end{proof}

We will now review the contents of~\cite[\S3.2]{voevodsky_homology} and extend
the results to
algebraic spaces and non-noetherian schemes.
We begin by recalling the construction of the sheaf associated to a presheaf,
cf.~\cite[Ch.~II, Thm.~2.11]{milne_etale_coh}.

\begin{definition}
Let $\sF$ be a presheaf on $\Sch_{/S}$ and equip $\Sch_{/S}$ with a
Grothendieck topology $\cT$. For any $V\in\Sch_{/S}$ we define an equivalence
relation $\sim$ on $\sF(V)$ where $f\sim g$ if there exist a covering
$\{\map{p_i}{U_i}{V}\}\in
\cT$ such that $p_i^*(f)=p_i^*(g)$ for every $i$. We let $\sF'$ be the
quotient of $\sF$ by this equivalence relation. Furthermore we let
$$\widetilde{\sF}=\varinjlim_{\cU} \check{H}^0(\cU,\sF')$$
where the limit is taken over all coverings $\cU=\{\map{p_i}{U_i}{V}\}\in
\cT$ and
$$\check{H}^0(\cU,\sF')=\ker\left(\equalizer{\prod_i \sF'(U_i)}
{\prod_{i,j} \sF'(U_i\times_V U_j)}\right)$$
is the \v{C}ech cohomology.
\end{definition}

\begin{remark}
It is easily seen that $\sF'$ is a separated presheaf. By~\cite[Lem.~II.1.4
(ii)]{artin_groth_top} it then follows that $\widetilde{\sF}$ is the
sheafification of $\sF$. Moreover, we have that $\sF'$ is the image presheaf
of $\sF$ by the canonical morphism $\sF\to\widetilde{\sF}$.
\end{remark}

\begin{definition}
Let $X$ be an algebraic space over $S$ and consider the representable presheaf
$h_X=\Hom_S(-,X)$ on $\Sch_{/S}$. Let $L'(X)=(h_X)'$ and $L(X)=\widetilde{h_X}$
be the separated presheaf and sheaf associated to $h_X$ in the $h$-topology. We
denote the corresponding notions in the $qfh$-topology by $L'_{qfh}(X)$ and
$L_{qfh}(X)$.
\end{definition}

\begin{lemma}[{\cite[Lem.~3.2.2]{voevodsky_homology}}]\label{L:L'-over-red}
Let $X$ be an algebraic space over $S$ and let $T$ be a reduced $S$-scheme.
Then $L'(X)(T)=L'_{qfh}(X)(T)=\Hom_S(T,X)$.
\end{lemma}
\begin{proof}
Let $\{U_i\to T\}$ be an $h$-covering. Then $\coprod_i U_i\to T$
is universally submersive and schematically dominant. It follows that
$\Hom_S(T,X)\to \prod_i \Hom_S(U_i,X)$ is injective
by Proposition~\pref{P:sch-dom+univ-sub:epimorphism}
\end{proof}

\begin{lemma}\label{L:L'-desc}\label{L:epimorphic-nature-of-coverings}
Let $X$ be an algebraic space \emph{locally of finite type} over $S$ and let
$T\in\Sch_{/S}$.
Then $L'(X)(T)=L'_{qfh}(X)(T)$ coincides with the image of
$$\Hom_S(T,X)\to \Hom_S(T_\red,X).$$
If $T'\to T$ is universally submersive, then $L(X)(T)\to L(X)(T')$ is
injective.
\end{lemma}
\begin{proof}
If two morphisms $\map{f,g}{T}{X}$ coincide after the composition with an
$h$-covering $\{U_i\to T\}$, then they coincide after composing with
$T_\red\to T$. Indeed, we have that $\coprod_i (U_i)_\red\to T_\red$ is an
epimorphism by Proposition~\pref{P:sch-dom+univ-sub:epimorphism}.
Conversely, we will show that if $f$ and $g$ coincide on $T_\red$ then they
coincide on a $qfh$-covering $T'\to T$.

Taking an open covering, we can assume that $T$ is affine. Let $\sN$ be the
sheaf of nilpotent elements of $\sO_T$, i.e., the ideal sheaf defining
$T_\red$. Then $\sN$ is the direct limit of its subsheaves of finite type.
%
Thus $T_\red$ is the inverse limit of
finitely presented nil-immersions $T_\lambda\inj T$. As $X\to S$ is locally of
finite type $\varinjlim_\lambda \Hom_S(T_\lambda,X)\to
\Hom_S(T_\red,X)$ is injective, cf.~\cite[Thm.~8.8.2]{egaIV}. Thus $f$ and $g$
coincide on $T_\lambda$ for some $\lambda$.

To show the last statement, it is enough to show that $L'(X)(T)\to L'(X)(T')$
is injective when $T'\to T$ is universally submersive. From the first part
of the lemma, it is thus enough to show that $\Hom_S(T_\red,X)\to \Hom_S(T'_\red,X)$ is injective and this is Proposition~\pref{P:sch-dom+univ-sub:epimorphism}.\end{proof}

\begin{remark}
Voevodsky claims that $L'(X)(T)=\Hom_S(T_\red,X)$ in the text
following~\cite[Lem.~3.2.2]{voevodsky_homology}. This is not correct
as $\Hom_S(T,X)\to\Hom_S(T_\red,X)$ need not be surjective. In fact, a
counter-example is given by $X=T_\red$ for any scheme $T$ such that $T_\red\inj
T$ does not have a retraction.
\end{remark}

\begin{proposition}\label{P:L-desc}
Let $X$ be an algebraic space locally of finite type over $S$, and let
$T\in\Sch_{/S}$. Then $L(X)(T)$ (resp.\ $L_{qfh}(X)(T)$) is the filtered direct
limit of
$$\ker\left(\equalizer{\prod_i X(U_i)}
{\prod_{i,j} X\bigl((U_i\times_T U_j)_\red\bigr)}\right)$$
where the limit is taken over all $h$-coverings (resp.\ $qfh$-coverings)
$\{U_i\to T\}$.
\end{proposition}
\begin{proof}
It is clear from the definitions of $L$ and $L'$ that $L(X)(T)$ is the limit of
$$\ker\left(\equalizer{\prod_i X(U_i)}
{\prod_{i,j} L'(X)\bigl(U_i\times_T U_j\bigr)}\right).$$
The proposition thus follows from Lemma~\pref{L:L'-desc}.
\end{proof}

In the remainder of this section, we can work in either the $h$-topology or the
$qfh$-topology, i.e., all instances of $L$ and $L'$ can be replaced with
$L_{qfh}$ and $L'_{qfh}$ respectively.

\begin{definition}
Let $X$ be an algebraic space over $S$, and let $T\in\Sch_{/S}$. Let $f\in
L(X)(T)$ be a section and let $\{\map{p_i}{U_i}{T}\}$ be a covering, i.e., a set
of morphisms such that $T=\bigcup_i p_i(U_i)$ but not necessarily an
$h$-covering. We say that $f$ is
\emph{realized} on the covering $\{p_i\}$ if there are morphisms $\{f_i\in
X(U_i)\}$ such that $p_i^*(f)=f_i$ in $L(X)(U_i)$ for every $i$.
\end{definition}

Let $\{\map{p_i}{U_i}{T}\}$ be a covering and let $\pi_1,\pi_2$ denote the
projections of $(U_i\times_T U_j)_\red$. If $f\in L(X)(T)$ is realized on
$\{p_i\}$ by $\{\map{f_i}{U_i}{X}\}$ then $f_i\circ \pi_1=f_j\circ \pi_2$ by
Lemma~\pref{L:L'-over-red}.
Conversely, if $X/S$ is \emph{locally of finite type} and $\{p_i\}$ is an
\emph{$h$-covering}, then morphisms $\{f_i\in X(U_i)\}$ such that
$f_i\circ \pi_1=f_j\circ \pi_2$, determines an element in $L(X)(T)$
by Proposition~\pref{P:L-desc}.

\begin{lemma}[{\cite[Lem.~3.2.6]{voevodsky_homology}}]
\label{L:realization-open-cov}
Let $X$ be an algebraic space over $S$, and let $T\in\Sch_{/S}$. Let $f\in
L(X)(T)$ and assume that $f$ is realized on an \etale{} covering
$\{\map{p_i}{U_i}{T}\}$. Then $f$ is realized on $T_\red$.
\end{lemma}
\begin{proof}
Let $\map{f_i}{U_i}{X}$ be a realization of $f$ on the covering $\{p_i\}$.
Then $f_i$ and $f_j$ coincide on
$(U_i\times_T U_j)_\red=(U_i\times_T U_j)\times_T T_\red$.
The $\{f_i\}$ thus glue to a morphism $T_\red\to X$ which realizes $f$.
\end{proof}

\begin{proposition}\label{P:realization-wn}
Let $X$ be an algebraic space locally of finite type over $S$ and let $T$ be an
$S$-scheme. Let $f\in L(X)(T)$ be a section. Then $f$ is realized on the
absolute weak normalization $\wn{}{T}$.
\end{proposition}
\begin{proof}
We can replace $T$ with $\wn{}{T}$ and assume that $T$ is weakly normal.
The section $f$ is realized on an $h$-covering of the form ${\{W_j\to V_j\to
T\}}$ where the $W_j\to V_j$ are quasi-compact $h$-coverings and $\{V_j\to T\}$
is an open covering. By Theorem~\pref{T:h-descent}, applied to the weakly
normal morphism $(W_j)_\red\to V_j$, and
Lemma~\pref{L:epimorphic-nature-of-coverings}, we have that $f$ is realized on
the covering $\{V_j\to T\}$.
Lemma~\pref{L:realization-open-cov} then shows that $f$ is realized on $T$.
\end{proof}

\begin{corollary}\label{C:realization-univ-homeo-fp}
Let $X$ be an algebraic space locally of finite presentation over $S$, and let
$T$ be a quasi-compact and quasi-separated $S$-scheme. Let $f\in L(X)(T)$ be a
section. Then $f$ is realized on a universal homeomorphism $U\to T$ of finite
presentation.
\end{corollary}
\begin{proof}
By Proposition~\pref{P:realization-wn} the section $f$ is realized on
$\wn{}{T}$. A limit argument shows that there is a finitely presented universal
homeomorphism $U\to T$ which realizes $f$.
\end{proof}

\begin{theorem}\label{T:htop-res}
Let $X$ be an algebraic space locally of finite presentation over $S$, and let
$T\in\Sch_{/S}$ be quasi-compact and quasi-separated. Then
$L(X)(T)=\varinjlim_\lambda \Hom_S(T_\lambda,X)=\Hom_S(\wn{}{T},X)$ where the
limit is taken over all finitely presented universal homeomorphisms
$T_\lambda\to T$.
\end{theorem}
\begin{proof}
If $T_\lambda\to T$ is a universal homeomorphism then $(T_\lambda\times_T
T_\lambda)_\red=(T_\lambda)_\red$. Thus, by Proposition~\pref{P:L-desc} we
obtain a canonical map
\begin{equation*}
\varinjlim_\lambda \Hom_S(T_\lambda,X)
\to L(X)(T).
\end{equation*}
The surjectivity of this map follows from Corollary~\pref{C:realization-univ-homeo-fp}.
To show injectivity, let $\map{f_1,f_2}{T_\lambda}{X}$ be two maps coinciding
in $L(X)(T)$. Then $f_1$ and $f_2$ coincide on $(T_\lambda)_\red$ and hence
also on $T_\mu$ for a finitely presented nil-immersion $T_\mu\inj
T_\lambda$. Finally, we have that
\begin{equation*}
\Hom_S(\wn{}{T},X)=\varinjlim_\lambda \Hom_S\left((T_\lambda)_\red,X\right)
=\varinjlim_\lambda \Hom_S\left(T_\lambda,X\right).\hfill\qedhere
\end{equation*}
\end{proof}

\begin{remark}
In the non-noetherian case, it may be useful to change the $h$-topology (resp.\
$qfh$-topology) to only require the coverings to be of finite type instead of
finite presentation. In particular $X_\red\to X$ would always be an
$h$-covering.
Then Lemma~\pref{L:L'-desc} holds without any assumptions on $X$ and we can
drop the assumption that $X/S$ is locally of finite type in
\ref{P:L-desc}--\ref{P:realization-wn}. The main
results~\ref{C:realization-univ-homeo-fp}--\ref{T:htop-res} remain valid for
this topology.
It is also likely that for this topology
Theorem~\pref{T:h-top-desc} holds if we let $W_j\to V_j$ be
any proper (resp.\ finite) surjective morphism, cf.\ Remark~\pref{R:st-non-fp}.
\end{remark}

\end{section}

\appendix
\begin{section}{\Etale{} morphisms and henselian pairs}\label{S:etale-henselian}
In this section, we first recall some facts about \etale{} morphisms which we
state in the category of algebraic spaces. We then consider schemes which are
proper over a local henselian scheme. Let $S$ be a henselian local ring with
closed point $S_0$, let $S'\to S$ be a proper morphism and let
$S'_0=S'\times_S S_0$. Then $(S',S'_0)$ is $0$-henselian (i.e., a henselian
couple) and $1$-henselian (i.e., induces an equivalence between finite \etale{}
covers). This is the key fact in the proof of the proper base change theorem in
\etale{} cohomology for degrees~$0$ and~$1$, cf.\ 
Theorem~\pref{T:0-1-henselian-for-proper-over-henselian}. We interpret these
henselian properties using algebraic spaces in
Proposition~\pref{P:hensel-equiv}. These results are the core of the proof that
proper morphisms are morphisms of effective descent for \etale{} morphisms,
cf.~Proposition~\pref{P:proper-descent-for-etfin} and
Corollary~\pref{C:proper-descent-for-etqc}.

The results~\pref{P:diag-of-etale-is-open}--\pref{P:nil-imm-equiv-etale} are
well-known for schemes. We indicate how to extend these results to algebraic
spaces:

\begin{proposition}[{\cite[Cor.~II.6.17]{knutson_alg_spaces}}]
\label{P:etale+sep-is-repr}
An \etale{} and separated morphism of algebraic spaces is representable.
\end{proposition}


\begin{proposition}\label{P:diag-of-etale-is-open}
Let $\map{f}{X}{Y}$ be an \etale{} morphism of algebraic spaces. Then:
\begin{enumerate}
\item $\Delta_f$ is an open immersion.
\item Any section of $f$ is an open immersion.\label{PI:etale-sect-open}
\item If $f$ is universally injective, then $f$ is an open immersion.
\label{PI:etale+univinj-is-open}
\end{enumerate}
\end{proposition}
\begin{proof}
(i) follows easily from the case where $X$ and $Y$ are schemes.
(ii) follows from (i) as any section of $f$ is a pull-back of $\Delta_f$.  For
(iii) we note that if $f$ is universally injective then $\Delta_f$ is
surjective. It follows by (i) that $f$ is separated and by
Proposition~\pref{P:etale+sep-is-repr} that $f$ is representable. We can thus
assume that $X$ and $Y$ are schemes.
\end{proof}

\begin{corollary}\label{C:WWW}
Let $X$ and $Y$ be algebraic spaces over $S$ such that $Y\to S$ is \etale{}.
There is then a one-to-one correspondence between morphisms $\map{f}{X}{Y}$
and open subspaces $\Gamma$ of $X\times_S Y$ such that $\Gamma\to X$ is
universally injective and surjective. This correspondence is given by mapping
$f$ to its graph $\Gamma_f$.
\end{corollary}
\begin{proof}
This follows immediately from~\ref{PI:etale-sect-open}
and~\ref{PI:etale+univinj-is-open} of
Proposition~\pref{P:diag-of-etale-is-open}.
\end{proof}

\begin{proposition}[{\cite[Thm.~18.1.2]{egaIV}}]\label{P:nil-imm-equiv-etale}
Let $S_0\inj S$ be a nil-immersion of schemes, i.e., a surjective closed
immersion. Then the functor $X\mapsto X\times_S S_0$ from the category of
\etale{} $S$-spaces (resp.\ $S$-schemes) to the category of \etale{}
$S_0$-spaces (resp.\ $S_0$-schemes) is an equivalence of categories.
\end{proposition}
\begin{proof}
That the functor is fully faithful follows from Corollary~\pref{C:WWW}. Let us
prove essential surjectivity. For the category of schemes, this follows
from~\cite[Thm.~18.1.2]{egaIV}. Let $X_0\to S_0$ be an \etale{} morphism of
algebraic spaces. Let $U_0\to X_0$ be an \etale{} presentation with a scheme
$U_0$. Then $R_0=U_0\times_{X_0} U_0$ is also a scheme. We thus obtain
$S$-schemes $R$ and $U$ and an \etale{} equivalence relation
$\equalizer{R}{U}$ which restricts to the equivalence relation given by
$R_0$ and $U_0$. The quotient $X$ of this equivalence relation restricts to
$X_0$.
\end{proof}


We recall two fundamental results for schemes which are proper over a complete
local ring.

\begin{proposition}\label{P:0-henselian-for-proper-over-complete}
Let $S$ be the spectrum of a noetherian complete local ring with closed point
$S_0$. Let $S'\to S$ be a \emph{proper} morphism and $S'_0=S'\times_S S_0$.
The map $W'\mapsto W'\cap S'_0$ is a bijection between the open and closed
subsets of $S'$ and the open and closed subsets of $S'_0$.
\end{proposition}
\begin{proof}
This is a special case of~\cite[Prop.~5.5.1]{egaIII}.
\end{proof}

\begin{theorem}[{\cite[Thm.~18.3.4]{egaIV}}]
\label{T:1-henselian-for-proper-over-complete}
Let $S$ be the spectrum of a noetherian complete local ring with closed point
$S_0$. Let $S'\to S$ be a \emph{proper} morphism and $S'_0=S'\times_S S_0$. The
functor $X'\mapsto X'\times_{S'} S'_0$ from the category of \etale{} and
finite $S'$-schemes to \etale{} and finite $S'_0$-schemes is then an
equivalence of categories.
\end{theorem}
\begin{proof}
Let $\widehat{S}$ and $\widehat{S'}$ be the completions of $S$ and $S'$ along
$S_0$ and $S'_0$ respectively. Grothendieck's existence
theorem~\cite[Thm.~5.1.4]{egaIII} shows that $X'\mapsto X'\times_{S'}
\widehat{S'}$ is an equivalence between the categories of finite \etale{}
covers of $S'$ and $\widehat{S'}$ respectively.
Proposition~\pref{P:nil-imm-equiv-etale} then shows that $\widehat{X'}\mapsto
\widehat{X'}\times_{\widehat{S'}} S'_0$ is an equivalence between covers of
$\widehat{S'}$ and covers of $S'_0$.
For details see~\cite[Thm.~18.3.4]{egaIV}.
\end{proof}

Using \etale{} cohomology, we get a nice interpretation of the above two
results:

\begin{proposition}\label{P:hensel-equiv}
Let $S$ be a quasi-compact and quasi-separated scheme. Let $S_0\inj S$
be a closed subscheme. If $F$ is a sheaf on the small \etale{} site on $S$,
then we let $F_0$ denote the pull-back of $F$ to $S_0$. Then
\begin{enumerate}
\item The following conditions are equivalent:
\begin{enumerate}
\myitemii{a} For any sheaf of sets $F$ on the small \etale{} site on $S$, the
canonical map
$$H^0_{\metale}(S,F)\to H^0_{\metale}(S_0,F_0)$$
is bijective.
\myitemii{a$'$} For any \emph{constructible} sheaf of sets $F$ on the small
\etale{} site on $S$, the canonical map
$$H^0_{\metale}(S,F)\to H^0_{\metale}(S_0,F_0)$$
is bijective.
\myitemii{b} For any finite morphism $S'\to S$, the map
$W'\mapsto W'\cap (S'\times_S S_0)$ from open and closed subsets of $S'$ to
open and closed subsets of $S'\times_S S_0$ is bijective.
\myitemii{c} For any \etale{} morphism of algebraic spaces $X\to S$ the
canonical map
$$\Gamma(X/S)\to \Gamma(X\times_S S_0/S_0)$$
is bijective.
\myitemii{c$'$} For any \etale{} \emph{finitely presented} morphism of
algebraic spaces $X\to S$ the canonical map
$$\Gamma(X/S)\to \Gamma(X\times_S S_0/S_0)$$
is bijective.
\end{enumerate}
\item The following conditions are equivalent:
\begin{enumerate}
\item For any sheaf $F$ of ind-finite groups on $S$, the canonical map
$$H^i_{\metale}(S,F)\to H^i_{\metale}(S_0,F_0)$$
is bijective for $i=0,1$.
\item The functor $X\mapsto X\times_{S} S_0$ from the category of \etale{} and
finite $S$-schemes to \etale{} and finite $S_0$-schemes is an equivalence of
categories.
\end{enumerate}
\end{enumerate}
\end{proposition}
\begin{proof}
Every sheaf of sets is the filtered direct limit of constructible
sheaves by~\cite[Exp.~IX, Cor.~2.7.2]{sga4}.
As $H^0_{\metale}$ commutes with filtered direct
limits~\cite[Exp.~VII, Rem.~5.14]{sga4},
the equivalence between (a) and (a$'$) follows. The equivalence between (a) and
(b) in (i) and (ii) is a special case of~\cite[Exp.~XII, Prop.~6.5]{sga4}.
For the equivalence between (a) and (c) in (i) we recall that there is an
equivalence between the category of sheaves on the small \etale{} site on $S$
with the category of algebraic spaces $X$ \etale{} over $S$, cf.~\cite[Ch.~V,
Thm.~1.5]{milne_etale_coh} or~\cite[Ch.~VII, \S1]{artin_theoremes_de_repr}.
This takes a sheaf to its ``espace \'{e}tal\'{e}'' and conversely an algebraic
space to its sheaf of sections. Furthermore, a sheaf is constructible if and
only if its espace \'{e}tal\'{e} is of finite
presentation~\cite[Exp.~IX, Cor.~2.7.1]{sga4}.

If $X\to S$ is an \etale{} morphism corresponding to the sheaf $F$, then
$H^0_{\metale}(S,F)=\Gamma(X/S)$. For any morphism $\map{g}{S'}{S}$, the
pull-back $g^{*}F$ is represented by $X\times_S S'$.
%
This shows that (a) and (c) as well as (a$'$) and (c$'$) are equivalent.
\end{proof}

\begin{remark}
If $S$ is not locally noetherian, then an espace \'{e}tal\'{e} need not be
quasi-separated. However, do note that any \etale{} morphism is \emph{locally
separated} by Proposition~\pref{P:diag-of-etale-is-open} and that finitely
presented morphisms are quasi-separated.
\end{remark}

\begin{remark}
Part~(i) of Proposition~\pref{P:hensel-equiv} is a generalization
of~\cite[Prop.~18.5.4]{egaIV} which only shows that (b) implies (c) for the
category of \emph{separated} \etale{} morphisms $X\to S$. An example of
Artin~\cite[Exp.~XII, Rem.~6.13]{sga4} shows that
condition (c) restricted to morphisms of \emph{schemes} does not always imply
(a) and (b). It does suffice when $S$ is affine though, cf.~\cite[Ch.~XI,
Thm.~1]{raynaud_hensel_rings}.
\end{remark}

\begin{definition}
Let $S$ be a quasi-compact and quasi-separated scheme and $S_0\inj S$ a closed
subscheme. We say that the pair $(S,S_0)$ is \emph{$0$-henselian} or
\emph{henselian} (resp.\ \emph{$1$-henselian}) if $(S,S_0)$ satisfies the
equivalent conditions of (i) (resp.\ (ii)) of Proposition~\pref{P:hensel-equiv}.
\end{definition}

We can now rephrase Proposition~\pref{P:0-henselian-for-proper-over-complete}
and Theorem~\pref{T:1-henselian-for-proper-over-complete} as follows:

\begin{theorem}
\label{T:0-1-henselian-for-proper-over-complete}
Let $S$ be the spectrum of a noetherian complete local ring with closed point
$S_0$. Let $S'\to S$ be a proper morphism and $S'_0=S'\times_S
S_0$. Then $(S',S'_0)$ is $0$-henselian and $1$-henselian.
\end{theorem}

Proposition~\pref{P:0-henselian-for-proper-over-complete} is easily extended
to noetherian henselian local rings using the connectedness properties of the
Stein factorization:

\begin{proposition}[{\cite[Prop.~18.5.19]{egaIV}}]
\label{P:0-henselian-for-proper-over-henselian}
Let $S$ be the spectrum of a noetherian \emph{henselian} local ring with closed
point $S_0$. Let $S'\to S$ be a proper morphism and $S'_0=S'\times_S
S_0$. Then $(S',S'_0)$ is $0$-henselian.
\end{proposition}

It is more difficult to show that $(S',S'_0)$ is $1$-henselian under the
assumptions of Proposition~\pref{P:0-henselian-for-proper-over-henselian} (and
we will not need this). One possibility is to use Artin's approximation
theorem. This is done in~\cite[Thm.~3.1]{artin_alg_approximation}. Another
possibility is to use Popescu's
theorem~\cite{swan_popescus_theorem,spivakovsky_popescus_theorem}.
As these powerful results were not available at the time, Artin gave an
independent proof in~\cite[Exp.~XII]{sga4}.
This result is also slightly more general as it does not require the proper
morphism to be finitely presented:

\begin{theorem}[{\cite[Exp.~XII, Cor.~5.5]{sga4}}]
\label{T:0-1-henselian-for-proper-over-henselian}
Let $S$ be the spectrum of a henselian local ring with closed point $S_0$. Let
$S'\to S$ be a proper morphism and $S'_0=S'\times_S S_0$. Then the pair
$(S',S'_0)$ is $0$-henselian and $1$-henselian.
\end{theorem}

Theorem~\pref{T:0-1-henselian-for-proper-over-henselian} is only part of the
full proper base change theorem in \etale{} cohomology~\cite[Exp.~XII,
Thm.~5.1, Cor.~5.5]{sga4}.
A slightly less general but easier proof of this theorem utilizing Artin's
approximation theorem and algebraic spaces can be found
in~\cite[Ch.~VII]{artin_theoremes_de_repr}.

\end{section}


\begin{section}{Absolute weak normalization}\label{S:weak-norm}
In this section, we introduce the \emph{absolute weak normalization}. This is
an extension of the \emph{weak normalization},
cf.~\cite{and-bom,manaresi,yanagihara_1}. The weak normalization (resp.\ 
absolute weak normalization) is dominated by the normalization (resp.\ total
integral closure). Recall that a separated universal homeomorphism $X'\to X$ of
algebraic spaces is the same thing as an integral, universally injective and
surjective morphism, cf.\ Corollary~\pref{C:sep-homeo-is-affine}
and~\cite[Cor.~18.12.11]{egaIV}.

\begin{definition}
A scheme or algebraic space $X$ is \emph{absolutely weakly normal} if
\begin{enumerate}
\item $X$ is reduced.
\item If $\map{\pi}{X'}{X}$ is a separated universal homeomorphism and
$X'$ is reduced, then $\pi$ is an isomorphism.
\end{enumerate}
If $X'\to X$ is a separated universal homeomorphism such that $X'$ is
absolutely weakly normal, then we say that $X'$ is an absolute weak
normalization of $X$.
\end{definition}

\begin{properties}\label{X:AWN}
We briefly list some basic properties of absolutely weakly normal schemes.
\begin{enumerate}
\item If $Y'\to Y$ is a separated universal homeomorphism and $X$ is absolutely
weakly normal, then any morphism $X\to Y$ factors uniquely through $Y'$. In
fact, ${(X\times_Y Y')_\red} \to X$ is an isomorphism. In particular, an
absolute weak normalization is unique if it exists.
\item The spectrum of a perfect field is absolutely weakly normal.
\item A TIC~scheme, cf.\ Definition~\pref{D:TIC}, is absolutely weakly
normal.
\item An absolutely flat scheme with perfect residue fields is absolutely
weakly normal. Every scheme $X$ has a canonical affine universally bijective
morphism $T^{-\infty}(X)\to X$ where $T^{-\infty}(X)$ is absolutely flat with
perfect residue fields~\cite{olivier_sem_samuel-2}.
\end{enumerate}
\end{properties}

We first establish the existence of the absolute weak normalization in the
affine case and then show that it localizes.

\begin{definition}
A ring extension $A\inj A'$ is called \emph{weakly subintegral} if
$\Spec(A')\to\Spec(A)$ is a universal homeomorphism. For an arbitrary extension
$A\inj B$, the \emph{weak subintegral closure} $\rwn{B}{A}$ of $A$ in $B$ is
the largest sub-extension $A\inj \rwn{B}{A}$ which is weakly subintegral. 
%
A ring $A$ is \emph{absolutely weakly normal} if its spectrum is absolutely
weakly normal. If $\Spec(A')\to\Spec(A)$ is an absolute weak normalization then
we say that $A'$ is the absolute weak normalization of $A$ and denote $A'$ with
$\rwn{}{A}$.
\end{definition}

Some comments on the existence of $\rwn{B}{A}$ are due. If $A\inj A'_1$ and
$A\inj A'_2$ are two weakly subintegral sub-extensions of $A\inj B$, then the
union $A'_1\cup A'_2=\image(A'_1\otimes_A A'_2\to B)$ is a weakly subintegral
sub-extension of $A\inj B$. If $(A'_i)$ is a filtered union of weakly
subintegral extension, then $A'=\bigcup_i A'_i$ is weakly
subintegral~\cite[Cor.~8.2.10]{egaIV}. The existence of $\rwn{B}{A}$ then
follows from Zorn's lemma.

\begin{properties}
The following properties are readily verified:
\begin{enumerate}
\item The weak subintegral closure is inside the integral
closure~\cite[Cor.~18.12.11]{egaIV}.
\item If $A\inj B$ is an extension and $B$ is absolutely weakly normal then
$\rwn{B}{A}$ is absolutely weakly normal.
\item If $A$ is an integral domain then the weak subintegral closure of $A$
in a perfect closure of its fraction field is the absolute weak normalization.
\item If $A$ is any ring then the weak subintegral closure of $A_\red$
in $\TIC(A_\red)$ (or $T^{-\infty}(A_\red)$) is the absolute weak normalization
$\rwn{}{A}$.
\end{enumerate}
\end{properties}

We have furthermore the following characterization of the weak subintegral
closure:

\begin{lemma}[{\cite[Thm.~(I.6)]{manaresi}}]\label{L:wn-char}
Let $A\inj B$ be an \emph{integral} extension. Then $b\in B$ is in the weak
subintegral closure $\rwn{B}{A}$ if and only if $b\otimes 1 = 1\otimes b$ in
$(B\otimes_A B)_\red$.
\end{lemma}
\begin{proof}
Let $A'=A[b]\subseteq B$. Then $\Spec(B)\to \Spec(A')$ is surjective and it
follows that $(A'\otimes_A A')_\red\to (B\otimes_A B)_\red$ is injective.
Thus $b\otimes 1 = 1\otimes b$ in $(B\otimes_A B)_\red$ if and only if
$(A'\otimes_A A')_\red \to A'_\red$ is an isomorphism. Equivalently, the
diagonal $\Delta_{\Spec(A')/\Spec(A)}$ is surjective which
by~\cite[Prop.~3.7.1]{egaI_NE} is equivalent to $\Spec(A')\to\Spec(A)$ being
universally injective. As $\Spec(A')\to\Spec(A)$ is finite and surjective,
$\Spec(A')\to\Spec(A)$ is universally injective if and only if
$\Spec(A')\to\Spec(A)$ is a universal homeomorphism. Thus $A\inj A[b]$ is
weakly subintegral if and only if $b\otimes1 = 1\otimes b$.
\end{proof}

\begin{proposition}\label{P:wn-basechange}
Let $A\inj B$ be an extension and let $A\to A'$ be a homomorphism. Assume that
$A\to A'$ is a localization or is \etale{}. Let $B'=B\otimes_A A'$. Then:
\begin{enumerate}
\item The weak subintegral closure $\rwn{B}{A}$ of $A$ in $B$ commutes with the
base change $A\to A'$, i.e., $\rwn{B'}{A'}=(\rwn{B}{A})\otimes_A A'$.
\item The absolute weak normalization $\rwn{}{A}$ of $A$ commutes with the
base change $A\to A'$, i.e., $\rwn{}{A'}=(\rwn{}{A})\otimes_A A'$.
\end{enumerate}
\end{proposition}
\begin{proof}
As the integral closure commutes with \etale{} base
change~\cite[Prop.~18.12.15]{egaIV}
and localizations, we can assume that $A\inj B$ is integral. By
Lemma~\pref{L:wn-char}, the sequence
$$\xymatrix{
\rwn{B}{A}\ar@{(->}[r] & B\ar@<.5ex>[r] \ar@<-.5ex>[r] & (B\otimes_A B)_\red
}$$
is exact. As exactness is preserved by flat morphisms and reduced rings are
preserved by localization and \etale{} base change,
cf.~\cite[Prop.~17.5.7]{egaIV}, it follows that
$\rwn{B'}{A'}=(\rwn{B}{A})\otimes_A A'$.

For the second part, let $B=\TIC(A)$ (or $B=T^{-\infty}(A)$). Then
$\rwn{}{A}=\rwn{B}{A}$ and in order to show that
$\rwn{}{A'}=(\rwn{}{A})\otimes_A A'$ it is enough to show that $B'$ is
absolutely weakly normal as $(\rwn{}{A})\otimes_A A'=\rwn{B'}{A'}$ by the first
part. Furthermore, it suffices to show that $B'_{\ip'}$ is absolutely weakly
normal for every prime $\ip'\in\Spec(B')$. Let $\ip$ be the image of $\ip'$ by
$\Spec(B')\to \Spec(B)$. Then $B_\ip\to B'_{\ip'}$ is essentially
\etale{}. But $B_\ip$ is strictly henselian, cf.\ Properties~\pref{X:TIC}, and
thus $B_\ip\to B'_{\ip'}$ is an isomorphism. As $B_\ip$ is a TIC~ring it is
absolutely weakly normal. If we instead use $B=T^{-\infty}(A)$ the last part of
the demonstration becomes trivial as $B_\ip$ and $B'_{\ip'}$ are perfect
fields.
\end{proof}

Let $S$ be a scheme or algebraic space. The proposition implies that given an
extension of quasi-coherent algebras $\sA\inj \sB$ on $S$, there is a unique
quasi-coherent sub-algebra $\rwn{\sB}{\sA}$ which restricts to the weak
subintegral closure on any affine covering. If $\map{\varphi}{\sA}{\sB}$ is not
injective but $\Spec(\sB)\to\Spec(\sA)$ is dominant, then we let
$\rwn{\sB}{\sA}$ be the weak subintegral closure of $\sA/\ker(\varphi)$
in~$\sB$.
Furthermore, there is a quasi-coherent sheaf of algebras
$\rwn{}{\sO_S}=\rwn{}{\sO_{S_\red}}$ and the spectrum of this algebra is the
absolute weak normalization of $S$. In the geometric case we adhere to the
notation in~\cite[Ch.~I, 7.2]{kollar_rat_curves_book}:

\begin{definition}
Let $S$ be a scheme or algebraic space. The \emph{weak normalization} of $S$
with respect to a quasi-compact and quasi-separated dominant morphism
$\map{f}{X}{S}$ is the spectrum of the weak subintegral closure of $\sO_S$ in
$f_*\sO_X$ and is denoted $\wn{X}{S}$. The \emph{absolute weak normalization}
of $S$ is denoted $\wn{}{S}$.
\end{definition}

\begin{remark}
An integral domain is said to be weakly normal if it is weakly normal in its
fraction field. Similarly, a reduced ring with a finite number of irreducible
components is weakly normal if it is weakly normal in its total fraction
ring~\cite{manaresi, yanagihara_2}. If $A$ is an excellent noetherian ring,
then its weak normalization is finite over $A$ and thus noetherian. The
absolute weak normalization on the other hand, need not be finite and may well
reside outside the category of noetherian rings.

There is also the notions of \emph{subintegral closure} and
\emph{semi-normality}~\cite{traverso,swan_seminormality,greco-traverso} which
coincide with weak subintegral closure and (absolute) weak normality in
characteristic
zero. The difference in positive characteristic is that $A\inj B$ is
subintegral if $\Spec(B)\to\Spec(A)$ is a universal homeomorphism with
\emph{trivial residue field extensions}, while weakly subintegral morphisms may
have purely inseparable field extensions. If $A$ is an excellent noetherian
ring then its semi-normalization is finite over $A$. In particular, if $A$ is
an excellent noetherian ring of characteristic zero, then the absolute weak
normalization, being equal to the semi-normalization, is finite over $A$.
\end{remark}

\end{section}

\bibliography{submersion}
\bibliographystyle{dary}

\end{document}